\documentclass[10pt,reqno,final]{article}

\usepackage{amsmath,amsfonts,amssymb,amsthm}
\usepackage{mathrsfs,fancybox,pifont}
\usepackage{graphicx,psfrag, float}
\usepackage{booktabs,multirow}
\usepackage{colortbl}
\usepackage{enumitem}
\usepackage{algorithm}
\usepackage[noend]{algpseudocode}
\usepackage{url}
\usepackage[notcite,notref]{showkeys}
\usepackage{color}
\usepackage{cases}
\usepackage{mathtools}
\usepackage{version}
\usepackage{authblk}
\usepackage{subcaption}
\usepackage{sectsty}

\newtheorem{theorem}{\textbf{Theorem}}[section]

\newtheorem{lemma}{\textbf{Lemma}}[section]

\newtheorem{remark}{\textbf{Remark}}[section]
\newtheorem{example}{\textbf{Example}}[section]

\allowdisplaybreaks

\sectionfont{\large}
\subsectionfont{\normalsize}
\title{Linear Relaxation Schemes with Asymptotically Compatible Energy Law for Time-Fractional Phase-Field Models}

\author{Hui Yu\thanks{School of Mathematics and Physics, University of Science and Technology Beijing, Beijing 100083, China (huiyu@xs.ustb.edu.cn).} \,,
Zhaoyang Wang\thanks{Corresponding author. School of Mathematical Sciences, Laboratory of Mathematics and Complex Systems, MOE, Beijing Normal University, Beijing 100875, China; \ Research Center for Mathematics, Beijing Normal University, Zhuhai, Guangdong 519088, China (zhaoyang584520@163.com).} \,,
Ping Lin\thanks{Corresponding author. Division of Mathematics, University of Dundee, Dundee DD1 4HN,United Kingdom (p.lin@dundee.ac.uk).}  }

\affil{}

\date{}

\begin{document}

\maketitle
\begin{abstract}
In this paper, we propose a variable time-step  linear relaxation scheme for time-fractional phase-field equations with a free energy density in general polynomial form. The $L1^{+}$-CN formula is used to discretize the fractional derivative, and an auxiliary variable is introduced to approximate the nonlinear term by directly solving algebraic equations rather than differential-algebraic equations as in the invariant energy quadratization (IEQ) and the scalar auxiliary variable (SAV) approaches. The developed semi-discrete scheme is second-order accurate in time, and the inconsistency between the auxiliary and the original variables does not deteriorate over time. Furthermore, we take the time-fractional volume-conserved Allen-Cahn equation, the time-fractional Cahn-Hilliard equation, and the time-fractional Swift-Hohenberg equation as examples to demonstrate that the constructed schemes are energy stable and that the discrete energy dissipation law is asymptotically compatible with the classical one when the fractional-order parameter $\alpha\rightarrow 1^{-}$. Several numerical examples demonstrate the effectiveness of the proposed scheme. In particular, numerical results confirm that the auxiliary variable remains well aligned with the original variable, and the error between them does not continue to increase over time before the system reaches steady state.


\medskip
\end{abstract}

\section{Introduction}
In recent years, the phase-field approach has been one of the most popular methodologies to simulate multiphase flows and phase‑separation processes. The phase-field model uses an order parameter $\phi$ to describe the interface of two fluids or the phase transition between two phases \cite{allen1979microscopic}, and $\phi$ is also employed to characterize the volume or mass fraction of a substance \cite{guo2015thermodynamically, guo2021diffuse}. It is regarded as a branch of gradient flows that evolve toward minimizing the free energy $E(\phi)$, and has been widely applied to materials science (See e.g. \cite{guo2020modeling}), multiphase flow (See e.g. \cite{guo2015thermodynamically, yang2025two}), and tumor growth modeling (See e.g. \cite{wang2025stability}).

The time-fractional phase-field model \cite{tang2019energy, liu2018time, liao2021energy} has recently attracted extensive interest due to its ability to capture memory effects and long-time scale behavior in phase separation processes. It replaces the classical integer‐order time derivative with a Caputo fractional derivative of order $0<\alpha<1$. In this work, we consider the time-fractional phase-field model:
\begin{equation}\label{the1-1}
 \begin{aligned}
 &\frac{\partial ^\alpha}{\partial t^\alpha} \phi= - M\mathcal{G} \mu, \\
 &\mu = \frac{\delta E}{\delta \phi}=\mathcal{L}\phi+F'(\phi),
 \end{aligned}
\end{equation}
where $M>0$ is the mobility coefficient and $\mathcal{G}$ is a semi-positive definite mobility operator.
The free energy is defined as
\begin{equation}
    E[\phi]=\int_\Omega\left(\frac{1}{2} \mathcal{L\phi}\cdot \phi + F(\phi)\right) d \mathbf{x}.
\end{equation}
The Caputo fractional derivative of $u$ is given by
\begin{equation}\nonumber
 \begin{aligned}
\frac{\partial ^\alpha}{\partial t^\alpha}u := \mathcal{I}_t^{1-\alpha} u'(t),
    \end{aligned}
\end{equation}
where $\mathcal{I}_t^{1-\alpha}$ is the Riemann-Liouville fractional integration operator defined as
\begin{equation}
  \mathcal{I}_t^{ \alpha} u  := \int_0^t \omega_\alpha(t-s)u(s)\, ds,\ \text{with}\ \omega_\alpha(t):=\frac{t^{\alpha-1}}{\Gamma(\alpha)} .
\end{equation}

Similar to the integer-order ($\alpha=1$) phase field equations, the time-fractional phase-field equations also follow certain energy laws. Tang \cite{tang2019energy} et al. proved that the energy is bounded above by the initial energy $E[t]\leq E[0]$ (for $t\in(0, T)$). To be compatible with the classical model, Liao et al. \cite{liao2021energy,jicahnhilliarden} proposed the variational energy dissipation law. That is, the variational energy functional $\mathcal{E}_\alpha[\phi]:=E[\phi] + \frac{M}{2}\mathcal{I}_t^\alpha \left\| \mu \right\|^2$ satisfies
\begin{equation}\label{variation_energy}
    \frac{d \mathcal{E}_\alpha[\phi] }{dt} + \frac{M}{2}\omega_\alpha(t)\left\| \mu \right\|^2 \leq 0,
\end{equation}
for the time-fractional Allen-Cahn equation. The variational energy functional $\mathcal{E}_\alpha [\phi]:=E[\phi] + \frac{M}{2}\mathcal{I}_t^\alpha \left\| \nabla\mu \right\|^2$ satisfies
\begin{equation}\label{variation_energy2}
    \frac{d \mathcal{E}_\alpha [\phi]}{dt} + \frac{M}{2}\omega_\alpha(t)\left\| \nabla\mu \right\|^2 \leq 0,
\end{equation}
for the time-fractional Cahn-Hilliard equation.

However, the variational energy $\mathcal{ {E}}_\alpha[\phi]$ lacks asymptotic compatibility with the Ginzburg-Landau energy $E(\phi)$ \cite{liao2025asymptotically,qi2024unified}. Recently, Quan et al. \cite{quan2022decreasing} defined a new modified energy
\begin{equation}\label{modified_energy}
\tilde{E}_\alpha[\phi] := E[\phi] + \frac{\omega_{1-\alpha}(t)}{2M} \| \phi(t) -  \phi(0)\| ^2 - \frac{1}{2M} \int_0^t \omega_{-\alpha}(t - s) \| \phi(t) -  \phi(s)\| ^2 \, ds,
\end{equation}
and preserve the following energy dissipation law
\begin{equation}\label{modified_energy_dissipation}
\frac{d\tilde{E}_\alpha [\phi]}{dt} - \frac{\omega_{-\alpha}(t)}{2M} \|\phi(t) - \phi(0)\| ^2 + \frac{1}{2M} \int_0^t \omega_{-\alpha-1}(t - s) \|\phi(t) - \phi(s)\| ^2 \, ds = 0,
\end{equation}
where
\begin{equation}
    \omega_{-\alpha}(t) = \frac{-\alpha}{t^{1+\alpha}\Gamma(1-\alpha)}, \quad \omega_{\alpha-1}(t) = \frac{\alpha(\alpha+1)}{t^{2+\alpha}\Gamma(1-\alpha)}.
\end{equation}
We can see that $\tilde{E}_\alpha[\phi] \rightarrow E[\phi]$ as $\alpha \rightarrow1^{-}$, which indicates that the new modified energy $\tilde{E}_\alpha[\phi]$ is asymptotically compatible with the Ginzburg–Landau energy $E[\phi]$.

To solve fractional-order phase-field models, many numerical methods proposed for integer-order phase-field models can be applied, such as the invariant energy quadratization (IEQ) method \cite{guillen2013linear, yang2016linear, yang2017numerical} and the scalar auxiliary variable (SAV) method \cite{shen2019new}. By combining the recently popular SAV method \cite{shen2019new}, Yu et al. \cite{yu2023exponential} constructed a linear, second-order accurate in time and energy bounded numerical scheme for the time-fractional Cahn-Hilliard equation. Qi and Zhao \cite{qi2024unified} constructed a CN-SAV scheme for the fractional gradient flow problem that satisfies the energy dissipation law (\ref{modified_energy_dissipation}) under a mild restriction on the step-size ratio. The IEQ and SAV methods require replacing the algebraic expression of the auxiliary variable directly by its time derivative to formulate an ordinary differential equation. However, this replacement of an algebraic expression by its direct time derivative is considered to have weak instability in the context of the differential-algebraic equation, that is, the original algebraic expression of the auxiliary variable will not be well maintained as time goes on (See e.g. \cite{lin1997sequential, alsafri2023numerical}). Such a temporal error increase for both IEQ and SAV methods has been observed in literature, e.g., \cite{jiang2022improving} and \cite{alsafri2023numerical}. Recently, based on an idea originally introduced by \cite{besse2002order} for the Schr\"{o}dinger equation, a linear relaxation method for the integer‑order phase‑field model with the double-well potential was independently proposed by \cite{zhang2024linear} and \cite{alsafri2023numerical}, the latter further considering the coupling with fluid equations and showing the reduced error in the algebraic expression as the time goes. The key idea is that the auxiliary variables are discretized on a time-staggered grid, resulting in the algebraic expression of the auxiliary variable in either IEQ or SAV being directly solved rather than replaced by its time derivative, to overcome the weak temporal instability known in the numerical treatment in the differential-algebraic equation context (See e.g. \cite{lin1997sequential, alsafri2023numerical}).

In this paper, we construct second-order variable steps numerical schemes for the time-fractional phase-field model with a free energy density $F(\phi)$ in general polynomial form, based on the $L1^{+}$-CN formula and a relaxation method. Notably, the proposed relaxation method differs from that in \cite{zhang2024linear} and \cite{alsafri2023numerical}, which appear to consider only the case where $F(\phi)$ is the double-well potential. The time-fractional volume-conserved Allen-Cahn equation, the time-fractional Cahn-Hilliard equation, and the time-fractional Swift-Hohenberg equation are taken as examples to show the corresponding numerical schemes. The energy stability and asymptotic compatibility of these schemes are proved. Furthermore, we illustrate through numerical results that the algebraic equation associated with the expression of the auxiliary variable in these schemes does not deteriorate over time, thus ensuring the numerical reliability in long-term computations. This particular advantage is reported here for the first time.

The remainder of this paper is organized as follows. In Section \ref{section2}, we develop the relaxation method for the general phase field model with a free energy density $F(\phi)$ in general polynomial form. In Section \ref{section3}, we use the proposed method to construct the $L1^+$-CN linear relaxation schemes and prove their energy stability and asymptotic compatibility for the time-fractional volume-conserved Allen–Cahn, Cahn–Hilliard, and Swift-Hohenberg equations. In Section \ref{section4}, we carry out numerical experiments to demonstrate the effectiveness of the proposed schemes. Some conclusions and remarks are given in Section \ref{section5}.

\section{Linear relaxation method with energy reformulation}
\label{section2}
We introduce a linear relaxation method to reformulate the free energy expression by introducing an auxiliary variable, thus obtaining time-independent auxiliary equations that allow the construction of numerical schemes in which the auxiliary and original variables are treated linearly.

To introduce the basic principles of the relaxation method, first, we give the free energy
\begin{equation}
    E[\phi] = \frac{1}{2} \left\langle \mathcal{L} \phi, \phi \right \rangle + \left\langle F(\phi),1\right \rangle,
\end{equation}
where $\langle \cdot, \cdot\rangle$ is the $L^2(\Omega)$ inner product, and the free energy density function $F(\phi)$ defined in a general form as
\begin{equation}
    F(\phi) =  \frac{a_1}{4}\phi^4 + \frac{a_2}{3} \phi^3 + \frac{a_3}{2} \phi^2 + a_4 \phi +a_5.
\end{equation}
$a_1 , a_2, a_3, a_4, a_5$ are constant coefficients and $a_1 >0$.

Let
\begin{equation}
    \begin{aligned}
b_1 &= a_1, \\
b_2 &= \frac{a_2}{2a_1}, \\
b_3 &= \frac{4a_1^2 a_3 - a_2^2}{8a_1^3}, \\
b_4 &= a_4 - \frac{a_2(4a_1^2 a_3 - a_2^2)}{8a_1^4}, \\
b_5 &= a_5 - \frac{(4a_1^2 a_3 - a_2^2)^2}{64a_1^6}.
\end{aligned}
\end{equation}
The functional potential energy $F[\phi ]$ can be reformulated as
\begin{equation}
F[\phi]  =  \int_\Omega \left(b_1  \phi^2 + b_2\phi  + b_3\right)^2  + b_4\phi +b_5 \ d \mathbf{x}.
\end{equation}

We introduce an auxiliary variable
\begin{equation}
    r (\phi) = b_1 \phi^2 + b_2 \phi + b_3 -S,
\end{equation}
where $S$ is the stabilized term.
The modified nonlinear potential term is equivalent to
\begin{equation}
    \tilde F[\phi ,r]  =\int_\Omega \left( r + S \right)^2 + b_4 \phi +b_5 \ d \mathbf{x},
\end{equation}
and
\begin{equation}
   \  F^{'} [\phi] = \int_\Omega 2(r+S)\cdot r' + b_4 \phi + b_5 \ d \mathbf{x},\ r' = 2b_1 \phi +  b_2.
\end{equation}
Then,   the phase-field model (\ref{the1-1}) can be rewritten as
\begin{equation}\label{relaxation_model}
    \begin{aligned}
    \partial_t^\alpha \phi & =  M \mathcal{G} \mu,\\
    \mu & = \mathcal{L}\phi + 2(r+S)\cdot r' + b_4 \phi + b_5,\\
    r& = b_1 \phi^2 + b_2 \phi + b_3 -S,
    \end{aligned}
\end{equation}
with the associated modified energy  defined as
\begin{equation}
    \tilde{E} [\phi, r] = \frac{1}{2} \left\langle \mathcal{L} \phi, \phi \right\rangle +  \left \langle \left( r + S \right)^2 + b_4 \phi +b_5, 1 \right \rangle.
\end{equation}
The relaxation system (\ref{relaxation_model}) obeys the modified energy dissipation law that
\begin{equation}
\begin{aligned}
  &\tilde{E}[T] - \tilde{E}[0]=\int_0^T \frac{d}{dt} \tilde{E}[\phi, r] dt = \int_0^T \int_{\Omega} \frac{\partial \tilde{E}}{\partial \phi} \phi_t + \frac{\partial \tilde{E}}{\partial r} r^{'} \phi_t  \, d\mathbf{x} dt \\
    &= \int_0^T \int_{\Omega} \left(\mathcal{L} \phi + b_4 \phi + b_5 + 2(r+S)r' \right) \phi_t \, d\mathbf{x}  dt\\
    &= -M^{-1}\int_0^T \int_{\Omega} (\mathcal{G}^{-1} \partial_t \phi) \phi_t \, d\mathbf{x}  dt    \leq 0.
\end{aligned}
\end{equation}
Based on the positive definiteness of the kernel function \cite[Corollary 2.1]{tang2019energy}, the above modified energy satisfies the energy boundedness law that
\begin{equation}\label{the7}
    \begin{aligned}
            &\tilde E[T]-\tilde E[0]  = -M^{-1}\int_0^T \int_\Omega\left( \mathcal{G} ^{-1} \partial_t^\alpha\phi\right)\phi_t d \mathbf{x} \, dt\leq 0.
    \end{aligned}
\end{equation}

This energy-boundedness law ensures the stability of the system.
To further capture the dissipative dynamics during the system evolution process, we introduce the associated modified variational energy following the idea of the decreasing upper bound functional of the energy in \cite{quan2022decreasing}. Define the variational modified energy
\begin{equation}
\tilde{E}_\alpha[\phi,r,q] := \tilde E[\phi,r,q] + \frac{\omega_{1-\alpha}(t)}{2M} \|\phi(t) - \phi(0)\| ^2 - \frac{1}{2M} \int_0^t \omega_{-\alpha}(t - s) \|\phi(t) - \phi(s)\| ^2 \, ds.
\end{equation}
It is established that
\begin{equation}
\frac{d\tilde{E}_\alpha[\phi,r,q]}{dt} - \frac{\omega_{-\alpha}(t)}{2M} \|\phi(t) - \phi(0)\| ^2 + \frac{1}{2M} \int_0^t \omega_{-\alpha-1}(t - s) \|\phi(t) - \phi(s)\| ^2 \, ds = 0.
\end{equation}
This variational modified energy $\tilde{E} \rightarrow E $ as $\alpha\rightarrow 1^{-} $, see \cite[Proposition 3.1]{quan2023decreasing}.

In the next section, we apply this relaxation method to three classical time-fractional phase-field models and construct the corresponding numerical schemes.

\section{The $L1^{+}$-CN  linear relaxation scheme}
\label{section3}
In this section, we first provide the $L1^{+}$-CN discrete formula for the fractional derivative and introduce an important lemma. We then construct numerical schemes for the time-fractional Allen-Cahn, Cahn-Hilliard, and Swift-Hohenberg equations and prove their energy stability and asymptotic compatibility.

For a given time $T \geq 0 $ and a positive integer $N$, set the non-uniform temporal grid $0 = t_0 < t_1 < \cdots<t_k<\cdots < t_N = T$, time-step sizes $\tau_k = t_k - t_{k-1}\ (1 \leq k \leq N)$, and the adjacent step size ratios $\rho_k=\frac{\tau_k}{\tau_{k-1}}\ (2 \leq k \leq N)$. For the grid function $u^k$, define the notation $\nabla_\tau u^k = u^k - u^{k-1}$, $\partial_\tau u^k = \nabla_\tau u^k/\tau_k$, and the midpoint approximation $u^{k-\frac{1}{2}} = \frac{u^k + u^{k-1}}{2}$.
Instead of directly evaluating $\partial_t^\alpha u$, the $L1^{+}$-CN scheme computes its average over the preceding time interval $\left[t_{n-1}, t_n \right]$, which makes the scheme better suited for non-uniform time steps. It is expressed as \cite{ji2020adaptive}:
\begin{equation}\label{L1+}
\left( \partial_{\tau}^{\alpha} u \right)^{n-\frac{1}{2}} := \frac{1}{\Gamma(1-\alpha) \tau_n} \int_{t_{n-1}}^{t_n} \sum_{k=1}^n \int_{t_{k-1}}^{\min\{t_k, t\}} (t - s)^{-\alpha} \frac{u^k - u^{k-1}}{\tau_k} \, ds \, dt = \sum_{k=1}^n b_{n-k}^{(1-\alpha, n)} \nabla_\tau u^k,
\end{equation}
where the discrete convolution kernel is given by
\begin{equation}\label{caputo_kernel}
b_{n-k}^{(\alpha, n)} = \frac{1}{\tau_n \tau_k} \int_{t_{n-1}}^{t_n} \int_{t_{k-1}}^{\min\{t_k, t\}} \omega_\alpha(t-s) \, ds \, dt.
\end{equation}

The following lemma will be used later to derive the asymptotic compatibility variational energy dissipation law for the linear relaxation scheme applied to three time-fractional phase-field models.

\begin{lemma}\cite{qi2024unified}\label{lemma_DGS}
    Under the time-step size ratio constraint
\begin{equation}\label{time_step_constraint}
\rho_{k+1} \geq H_\alpha\left(\rho_k\right)=\left[ \frac{  2h(\rho_k) - h(2\rho_k)  }{ \rho_k^{ \alpha} (4 - 2^{1+\alpha}) } \right]^{\frac{1}{ \alpha}}, \quad   \text{for} \ k \geq 2,
\end{equation}
where $h(s) = (1 + s)^{1+\alpha} - s^{1+\alpha} - 1$. Then the modified discrete kernels $\tilde b_{n-k}^{(\alpha, n)}$ satisfy the following relationships :
\begin{enumerate}
    \item $\tilde{b}^{(\alpha,n)}_{n-k-1} \geq \tilde{b}^{(\alpha,n)}_{n-k}$,   $1\leq k\leq n-1$,
    \item $\tilde{b}^{(\alpha,n-1)}_{n-k-1} \geq \tilde{b}^{(\alpha,n)}_{n-k}$,   $1\leq k\leq n-1$,
    \item $ \tilde{b}^{(\alpha,n-1)}_{n-2-k} - \tilde{b}^{(\alpha,n-1)}_{n-1-k} \geq \tilde{b}^{(\alpha,n)}_{n-1-k} - \tilde{b}^{(\alpha,n)}_{n-k} $,  $1\leq k\leq n-2$.
\end{enumerate}
with modified kernels
\begin{equation}\label{modified_dis_kernel}
    \tilde{b}^{(\alpha,k)}_0 = 2b^{(\alpha,k)}_0,\ \tilde{b}^{(\alpha,k)}_{k-1} = b^{(\alpha,k)}_{k-1},\ 1 \leq j\leq k-1.
\end{equation}
Furthermore, the following discrete gradient structure of the $L1^+$ scheme (\ref{L1+}) holds
\begin{equation}\label{DGS}
\left\langle \sum_{k=1}^n b_{n-k}^{(\alpha, n)} u^k, u^n \right\rangle  = \mathcal{A}_a\left(u^n\right) - \mathcal{A}_a\left(u^{n-1}\right) + \mathcal{R}_a\left(u^n\right),\quad \text{for}\ n \geq 2,
\end{equation}
where the two norm functionals are nonnegative and defined by
\begin{equation} \label{GS_kernal1}
\mathcal{A}_\alpha\left(u^n\right) = \frac{1}{2} \sum_{k=1}^{n-1} \left( \tilde{b}^{(\alpha,n)}_{n-k-1} - \tilde{b}^{(\alpha,n)}_{n-k} \right) \left\| \sum_{j=k+1}^n u^j \right\| ^2 + \frac{1}{2}  \tilde{b}^{(\alpha,n)}_{n-1}\left\| \sum_{j=1}^n u^j \right\| ^2,
\end{equation}

\begin{equation} \label{GS_kernal2}
\mathcal{R}_a\left(u^n\right) = \frac{1}{2} \sum_{k=1}^{n-2} \left( \tilde{b}^{(\alpha,n-1)}_{n-2-k} - \tilde{b}^{(\alpha,n-1)}_{n-1-k} - \tilde{b}^{(\alpha,n)}_{n-1-k} + \tilde{b}^{(\alpha,n)}_{n-k} \right) \left\| \sum_{j=k+1}^{n-1} u^j \right\| ^2 + \frac{1}{2} \left( \tilde{b}^{(\alpha,n-1)}_{n-2} - \tilde{b}^{(\alpha,n)}_{n-1} \right) \left\| \sum_{j=1}^{n-1} u^j \right\| ^2.
\end{equation}
\end{lemma}

In the following subsections, three exemplary time-fractional phase-field models will be used to illustrate the implementation of our framework.
\subsection{Model I: Volume-conserved time-fractional Allen-Cahn equation}

Consider the volume-conserved time-fractional Allen-Cahn equation
\begin{equation}\label{TFAC}
    \partial_t^\alpha \phi = -M \left(\mu - \mathcal{\eta} \right),\ \mu = \frac{\delta E}{\delta \phi},\ \text{and}\ \mathcal{\eta}(t) = \frac{1}{|\Omega|}\left( \mu,1\right),
\end{equation}
where the free energy is defined as
\begin{equation}
    E[\phi] =\int_\Omega  \frac{\varepsilon^2}{2}|\nabla \phi|^2 + \frac{1}{4}\left(\phi^2-1\right)^2 d\mathbf{x}.
\end{equation}
$\mathcal{\eta}$ is the Lagrange multiplier to keep the Allen-Cahn model volume-conserved, and $\varepsilon>0$ is an interface width parameter.

By introducing an auxiliary variable,
\begin{equation}
    r(\phi) = \phi^2-1-S,
\end{equation}
Eq.~(\ref{TFAC}) can be reformulated in the form of
\begin{subequations}\label{RRER_TFAC}
    \begin{align}
        &\partial_t^\alpha \phi  = - M \left(\mu - \bar{\mu} \right) ,\label{RRER_TFAC1} \\
        &\mu = -\varepsilon^2 \Delta \phi +  \left(r + S \right) \phi, \label{RRER_TFAC2}\\
        &r= \phi^2-1 - S,\label{RRER_TFAC3}\\
        &\bar{\mu} = |\Omega|^{-1}\langle \mu,1 \rangle.\label{RRER_TFAC4}
    \end{align}
\end{subequations}
Define the associated modified energy
\begin{equation}\label{modienergyac}
    \tilde {E} [\phi,r] := \int_\Omega \left(\frac{ \varepsilon^2}{2}|\nabla \phi|^2 + \frac{1}{2} (r+S)\left(\phi^2-1 -S \right) -\frac{1}{4}r^2\right) \ d \mathbf{x} + \frac{S^2}{4}|\Omega|,
\end{equation}
and the modified variational energy
\begin{equation}
\tilde{E}_\alpha[\phi,r] := \tilde E[\phi,r] + \frac{\omega_{1-\alpha}(t)}{2M} \|\phi(t) - \phi(0)\| ^2 - \frac{1}{2M} \int_0^t \omega_{-\alpha}(t - s) \|\phi(t) - \phi(s)\| ^2 \, ds.
\end{equation}

By substituting Eq.~(\ref{RRER_TFAC3}) into Eq.~(\ref{modienergyac}), the modified energy $\tilde{E}[\phi, r]$ is demonstrated to be consistent with the original energy $E[\phi]$. Consequently, for a given initial condition $\phi_0$, the relation $\tilde{E}[t_0] = E[t_0]$ holds. It is proven that
\begin{equation}\label{modified_energy_dt}
    \begin{aligned}
        &\frac{d \tilde E[\phi,r]}{dt} = \int_\Omega\frac{\delta \tilde E}{\delta \phi}\phi_t + \frac{\delta \tilde E}{\delta r}r'\phi_t d \mathbf{x}\\
        &=\int_\Omega \left(\varepsilon^2 \Delta\phi +\frac{1}{2} (r+S)\cdot2\phi+\frac{1}{2}r^{'}\cdot \left(\phi^2-1-S \right)-\frac{1}{2}rr' \right) \phi_t d\mathbf{x} \\
        &= -M^{-1}\int_\Omega\left(  \partial_t^\alpha\phi\right)\phi_t d \mathbf{x}=\frac{dE}{dt}.
    \end{aligned}
\end{equation}
Thus, the following relationship holds.
\begin{equation}\label{the7}
    \begin{aligned}
            &\tilde E[\phi(T), r(T)] - \tilde E[\phi(0), r(0)]  = -M^{-1}\int_0^T \int_\Omega\left(  \partial_t^\alpha\phi\right)\phi_t d \mathbf{x} \, dt\leq 0.
    \end{aligned}
\end{equation}
Furthermore, the modified variational energy satisfies that
\begin{equation}
\frac{d\tilde{E}_\alpha[\phi,r]}{dt} - \frac{\omega_{-\alpha}(t)}{2\kappa} \|\phi(t) - \phi(0)\| ^2 + \frac{1}{2\kappa} \int_0^t \omega_{-\alpha-1}(t - s) \|\phi(t) - \phi(s)\| ^2 \, ds = 0,
\end{equation}
where $\| \cdot\|$ is the $L^2$ norm.

Then, we propose the semi-discrete scheme. Integrating Eq.~(\ref{RRER_TFAC}) over $[t_{n}, t_{n+1}], n\geq 0$, and using the $L1^{+}$ formula for the Caputo derivative, we obtain the linear relaxation $L1^+$-CN scheme as follows
\begin{subequations}\label{Dis_RRER_AC}
    \begin{align}
        &\partial_\tau^\alpha \phi^{n+\frac{1}{2}} = - M\left(\mu^{n+\frac{1}{2}} - \bar{\mu} ^{n+\frac{1}{2}} \right), \label{Dis_RRER_AC1} \\
        &\mu^{n+\frac{1}{2}} = -\varepsilon^2 \Delta \phi^{n+\frac{1}{2}} +  \left(r^{n+\frac{1}{2}} + S \right) \phi^{n+\frac{1}{2}} ,\label{Dis_RRER_AC2}\\
        &\frac{r^{n+\frac{1}{2}} +r^{n-\frac{1}{2}} }{2} = \left(\phi^n \right)^2-1 - S,\label{Dis_RRER_AC3}\\
        &\bar{\mu}^{n+\frac{1}{2}} = |\Omega|^{-1} \langle \mu^{n+\frac{1}{2}},1 \rangle.
    \end{align}
\end{subequations}
The discrete modified energy is
\begin{equation}\label{modienergycac}
\begin{split}
    &\tilde {E} \left [ \phi^{n+1},r^{n+\frac{1}{2}} \right ]\\
    & = \int_\Omega \left(\frac{ \varepsilon^2}{2}|\nabla \phi^{n+1}|^2 + \frac{1}{2} (r^{n+\frac{1}{2}}+S)\left(\left(\phi^{n+1}\right)^2-1 -S \right) -\frac{1}{4}\left(r^{n+\frac{1}{2}} \right)^2\right) \ d \mathbf{x} + \frac{S^2}{4}|\Omega|.
    \end{split}
\end{equation}
For $n=0$, we set $r^{\frac{1}{2}} = ( \phi^0)^2 -1 - S$, and extend this definition to include $r^{-\frac{1}{2}} =  ( \phi^0)^2 -1 - S$. Consequently, for a given initial value $\phi^0$, the modified energy satisfies $\tilde{E}\left[ \phi^0, r^{-\frac{1}{2}} \right] = E\left[ \phi^0 \right]$.

\begin{theorem}\label{theorem_dis_AC_volm_consr}
The $L1^+$-CN  linear relaxation scheme (\ref{Dis_RRER_AC}) preserves the following discrete volume-conservation law
\begin{equation}
    \left \langle \phi^n, 1 \right \rangle =  \left \langle \phi^0, 1 \right \rangle,\ n \geq 1.
\end{equation}
\end{theorem}
\begin{proof}
    We will prove the theorem by mathematical induction. It can be verified by a direct computation that $\left \langle \phi^1,1\right \rangle = \left \langle \phi^0,1 \right \rangle$. Assume that $\left \langle \phi^k,1\right \rangle = \left\langle \phi^0,1 \right \rangle, (k=1,\cdots,n-1)$ holds. For $k=n$, according to Eq.~(\ref{Dis_RRER_AC}), we have
    \begin{equation}
        \begin{aligned}
            \left \langle \phi^n - \phi^{n-1},1 \right \rangle&= - \frac{M}{b_0^{(\alpha,n)} }\left( \langle \mu^{n-\frac{1}{2}} - \bar{\mu} ^{n-\frac{1}{2}},1 \right \rangle - \sum_{k=1}^{n-1}\frac{b_{n-k}^{(\alpha, n)}}{b_0^{(\alpha,n)} }\left \langle \phi^k - \phi^{k-1},1 \right \rangle,\\
            & = - \frac{M}{b_0^{(\alpha,n)} } \left(\left\langle\mu^{n-\frac{1}{2}} ,1 \right\rangle -\bar{\mu} ^{n-\frac{1}{2}} \right)=0.
        \end{aligned}
    \end{equation}
    The proof is complete.
\end{proof}

\begin{theorem}\label{theorem_energy_TFAC}
The $L1^+$-CN  linear relaxation scheme (\ref{Dis_RRER_AC}) preserves the modified energy law
    \begin{equation}\label{disenergylawac}
        \tilde E\left[ \phi^{n+1} , r^{n+\frac{1}{2}}\right]\leq   \tilde E\left[ \phi^{0} , r^{-\frac{1}{2}}\right]=  E\left[ \phi^0\right],\ \text{for}\ n\geq 0,
    \end{equation}
where $\tilde{E}\left[ \phi^{n+1}, r^{n+1/2}\right]$ is defined in (\ref{modienergycac}).
Under the time-step ratio constraint (\ref{time_step_constraint}), this scheme preserves the modified variational energy dissipation property that
\begin{equation} \label{varia_dis_ener_ac}
\tilde{E}_\alpha \left[\phi^{n+1}, r^{n+ \frac{1}{2}}\right]-\tilde{E}_\alpha \left[\phi^{n}, r^{n- \frac{1}{2}}\right]=- \frac{1}{M } \mathcal{R}_{1-\alpha} \left( \nabla_\tau \phi^{n+1}\right) \leq 0,\ \text{for}\ n \geq 2,
\end{equation}
with
\begin{equation}
\tilde{E}_\alpha \left[ \phi^{n+1}, r^{n+\frac{1}{2}} \right]:=\tilde{E}\left[ \phi^{n+1}, r^{n+\frac{1}{2}} \right] +\frac{1}{M  } \mathcal{A}_{1-\alpha}\left( \nabla_\tau \phi^{n+1}\right).
\end{equation}
\end{theorem}

\begin{proof}
    Taking $L^2$ inner product of Eq. \eqref{Dis_RRER_AC1} and Eq. \eqref{Dis_RRER_AC2} with $\nabla_\tau\phi^{k+1} $ for $1 \leq k \leq n$, and summing yields
    \begin{equation} \label{proofcac}
        -M^{-1}\left\langle \partial_\tau^\alpha \phi^{k+ \frac{1}{2}},\nabla_\tau\phi^{k+1}\right \rangle = \left\langle \mu^{k+\frac{1}{2}} -  \bar{\mu}^{k+\frac{1}{2}},  \nabla_\tau\phi^{k+1}\right \rangle.
    \end{equation}

Notice that the following relationship holds
\begin{equation}\label{ac_relax1}
    \begin{aligned}
&(r^{k+\frac{1}{2}} + S) \frac{\phi^{k} + \phi^{k+1}}{2} \nabla_{\tau} \phi^{k+1}\\
&= \frac{1}{2} (r^{k+\frac{1}{2}} +S) \left[ (\phi^{k+1})^2 - (\phi^{k})^2 \right]
- \frac{1}{2} (r^{k-\frac{1}{2}} + S) \left[ (\phi^{k})^2 - (\phi^{k})^2 \right]  \\
&= \frac{1}{2} (r^{k+\frac{1}{2}} + S) (\phi^{k+1})^2 - \frac{1}{2} (r^{k-\frac{1}{2}} + S) (\phi^{k})^2
- \frac{1}{2} (r^{k+\frac{1}{2}} - r^{k-\frac{1}{2}}) (\phi^{k})^2.
\end{aligned}
\end{equation}
For the right term of Eq.~(\ref{proofcac}),
\begin{equation}\label{ac_relax2}
    \begin{aligned}
&\left \langle \mu^{k+\frac{1}{2}}- \bar{\mu}^{k+\frac{1}{2}}, \nabla_{\tau} \phi^{k+1} \right \rangle = \left \langle\mu^{k+\frac{1}{2}},\nabla_{\tau} \phi^{k+1} \right \rangle -\left \langle\bar{\mu}^{k+\frac{1}{2}},\nabla_{\tau} \phi^{k+1} \right \rangle\\
&= \left \langle \varepsilon^2 \Delta \frac{\phi^{k} + \phi^{k+1}}{2} + (r^{k+\frac{1}{2}} +S) \frac{\phi^{k} + \phi^{k+1}}{2}, \nabla_{\tau} \phi^{k+1} \right \rangle \\
&= \frac{\varepsilon^2}{2} \left \langle \|\nabla \phi^{k+1}\|^2_{L^2} - \|\nabla \phi^{k}\|^2_{L^2} \right \rangle
+ \left \langle (r^{k+\frac{1}{2}} +S) \frac{\phi^{k} + \phi^{k+1}}{2}, \nabla_{\tau} \phi^{k+1} \right \rangle \\
&= \frac{\varepsilon^2}{2} \left \langle \|\nabla \phi^{k+1}\|^2_{L^2} - \|\nabla \phi^{k}\|^2_{L^2} \right \rangle
+ \frac{1}{2} \left \langle (r^{k+\frac{1}{2}} +S) \left(  (\phi^{k+1})^2-1-S\right), 1 \right \rangle \\
&\quad - \frac{1}{2} \left \langle (r^{k-\frac{1}{2}} +S) \left( (\phi^{k})^2-1 -S\right), 1 \right \rangle
- \frac{1}{4} \left \langle (r^{k+\frac{1}{2}})^2 - (r^{k-\frac{1}{2}})^2, 1 \right \rangle \\
&= \tilde{E}\left [ \phi^{k+1}, r^{k+\frac{1}{2}}\right] - \tilde{E}\left[ \phi^{k }, r^{k-\frac{1}{2}}\right ].
\end{aligned}
\end{equation}
Note that the discrete volume-conserved law $\left \langle\bar{\mu},\ \nabla_\tau\phi^{k+1} \right\rangle = \bar{\mu} \left \langle 1,\ \nabla_\tau \phi^{k+1} \right \rangle =0$ used here.
By combining Eq.~(\ref{ac_relax1}) and Eq.~(\ref{ac_relax2}), we obtain
    \begin{equation}\label{Dis_rRER_AC2_ener}
        \tilde E\left[\phi^{k+1},r^{k+\frac{1}{2}}\right] - \tilde E\left[\phi^{k},r^{k-\frac{1}{2}}\right] = -\frac{1}{M}\left\langle  \partial_\tau^\alpha \phi^{k+\frac{1}{2}},\nabla_\tau\phi^{k+1} \right\rangle  .
    \end{equation}
By summing Eq.~(\ref{Dis_rRER_AC2_ener}) from $k=0$ to $n$, the discrete energy boundedness law (\ref{disenergylawac}) is established.

Under the time-step ratio constraint condition, we have
    \begin{equation}\label{dis_GS_ac}
        \left( \partial_\tau^\alpha \phi^{n+\frac{1}{2}},\nabla_\tau\phi^{n+1} \right) = \mathcal{A}_{1-\alpha} \left(\nabla_\tau \phi^{n+1} \right) - \mathcal{A}_{1-\alpha} \left(\nabla_\tau \phi^n \right) + \mathcal{R}_{1-\alpha} \left(\nabla_\tau \phi^{n+1} \right).
    \end{equation}
According to Eq.~(\ref{Dis_rRER_AC2_ener}) and Eq.~(\ref{dis_GS_ac}), there is
\begin{equation}
 \tilde E_\alpha\left[ \phi^{n+1},r^{n+\frac{1}{2}}\right] -  \tilde E_\alpha\left[\phi^{n},r^{n-\frac{1}{2}}\right]   = -\frac{1}{M }\mathcal{R}_{1-\alpha}\left(\nabla
 _\tau \phi^n\right),
\end{equation}
with
$\mathcal{A}_\alpha(\cdot)$ and $\mathcal{R}_\alpha \left( \cdot \right)$ are defined in (\ref{GS_kernal1}) and (\ref{GS_kernal2}).

This concludes the proof.
\end{proof}
\begin{remark}\label{remark2.1}
We establish below the connection between the modified energy formulation and the original energy expression. Note that
\begin{equation}\label{remarkac}
    \begin{aligned}
        &E\left[\phi^{n+1} \right] \approx \tilde E\left[ \phi^{n+1}, r^{n+1}\right]\\
         &= \int_\Omega \left(\frac{\varepsilon^2}{2}|\nabla \phi^{n+1}|^2 + \frac{1}{2} \left(r^{n+1} +S\right)\left( \left(\phi^{n+1}\right)^2 -1  - S \right) -\frac{1}{4}\left( r^{n+1} \right)^2\right) \ d \mathbf{x} + \frac{S^2}{4}|\Omega|\\
        &= \int_\Omega \left(\frac{\varepsilon^2}{2}|\nabla \phi^{n+1}|^2 + \frac{1}{2} \left(r^{n+\frac{1}{2}} +S\right)\left( \left(\phi^{n+1}\right)^2 -1  - S \right) -\frac{1}{4}\left( r^{n+\frac{1}{2}} \right)^2\right) \ d \mathbf{x} + \frac{S^2}{4}|\Omega|\\
        &\quad +\int_\Omega\frac{1}{2}\left( r^{n+1}-r^{n+\frac{1}{2}} \right)\left( \left(\phi^{n+1}\right)^2 -1  - S \right) -\frac{1}{4}\left( \left( r^{n+1} \right)^2-  \left( r^{n+\frac{1}{2}} \right)^2\right) d\mathbf{x}\\
        & \approx \tilde E\left[\phi^{n+1}, r^{n+\frac{1}{2}}\right]+\frac{1}{4}\left\|r^{n+1}-r^{n+\frac{1}{2}}  \right\|_{L^2}^2.
    \end{aligned}
\end{equation}
$\left\|r^{n+1}-r^{n+\frac{1}{2}}  \right\|_{L^2}^2$ are second-order accurate in time. In the discrete case, the modified energy is a second-order approximation of the original energy.
\end{remark}

 \begin{remark}\label{remark 2.2}
     (Asymptotic compatibility)
    As fractional order $\alpha \rightarrow 1^{- }$, the fractional discrete system (\ref{Dis_RRER_AC}) degenerates to the integral CN scheme that
     \begin{equation}
    \partial_\tau  \phi^n = - M\left(\mu^{n-\frac{1}{2}} - \bar{\mu} ^{n-\frac{1}{2}} \right),
\end{equation}

Notice that the discrete convolution kernels and the modified discrete convolution kernels satisfy $ {b}_0^{(1-\alpha, n)} \xrightarrow{\alpha \rightarrow 1^{-}} {\tau_n}^{-1},\ \tilde{b}_0^{(1-\alpha, n)} \xrightarrow{\alpha \rightarrow 1^{-}} 2{\tau_n}  ^{-1}$ and $ {b}_j^{(1-\alpha, n)} \xrightarrow{\alpha \rightarrow 1^{-}} 0, \  \tilde{b}_j^{(1-\alpha, n)} \xrightarrow{\alpha \rightarrow 1^{-}} 0\ (1 \leq j \leq n-1)$. The discrete gradient structure constructed terms satisfy $\mathcal{A}_{1-\alpha}\left[ \nabla_\tau \phi^n\right] \xrightarrow{\alpha \rightarrow 1^{-}} \tau_n^{-1}\left \| \nabla_\tau \phi^n \right\|^2  = \tau_n \left\| \partial_\tau \phi^n \right \|^2$, and $\mathcal{R}_{1-\alpha}\left[ \nabla_\tau \phi^n\right] \xrightarrow{\alpha \rightarrow 1^{-}}   \tau_n \left\| \partial_\tau \phi^n \right \|^2 $.
Thus, as $\alpha \rightarrow 1^-$, the variational energy dissipation (\ref{varia_dis_ener_ac}) becomes
\begin{equation}\label{acintegralenergy}
    \partial_\tau \tilde E\left[\phi^n , r^{n-\frac{1}{2}}\right] = -M^{-1}\left\| \partial_\tau \phi^n \right\|^2  = -M \left\| \mu^{n-\frac{1}{2}} - \bar{\mu}^{n-\frac{1}{2}}  \right\|^2, \ 1 \leq n \leq N.
\end{equation}
It indicates that the discrete variational energy dissipation law (\ref{varia_dis_ener_ac}) is asymptotically compatible with the integral discrete energy dissipation law as $\alpha \rightarrow 1^{-}$.
 \end{remark}

\subsection{Model II: Time-fractional Cahn-Hilliard equation}
In this section, we consider the time fractional Cahn-Hilliard equation.
\begin{equation}\label{TFCH}
    \partial_t^\alpha \phi = M \Delta \mu,\ \mu = \frac{\delta  E}{\delta \phi},
\end{equation}
with the free energy
\begin{equation}
    E[\phi] =\int_\Omega \frac{\varepsilon^2}{2}|\nabla \phi|^2 + \frac{1}{4}\phi^2(1-\phi)^2 \ d\mathbf{x}.
\end{equation}

By introducing an auxiliary variable,
\begin{equation}
     r  = \phi(1-\phi) - S,
\end{equation}
the time fractional phase-field model (\ref{TFCH}) is rewritten as
\begin{subequations}\label{relaxation_ch}
    \begin{align}
        \partial_t^\alpha \phi &= M \Delta \mu,\\
        \mu &= -\varepsilon^2 \Delta \phi + \frac{1}{2}\left(r + S \right)(1-2\phi),\\
        r & = \phi(1-\phi) - S.\label{relaxation_ch3}
    \end{align}
\end{subequations}
The corresponding modified energy is reformulated in the form of
\begin{equation}\label{modienergy}
    \tilde {E} [\phi,r] := \int_\Omega \frac{ \varepsilon^2}{2}|\nabla \phi|^2 + \frac{1}{2} ( r + S )\left(\phi\left( 1-\phi\right) - S \right) -\frac{1}{4}r^2 \ d \mathbf{x} + \frac{S^2}{4}|\Omega|.
\end{equation}
It can be verified that the modified energy $\tilde{E}[\phi, r]$ is demonstrated to be consistent with the original energy $E[\phi]$, and the relationship $\tilde{E}[t_0] = E[t_0]$ holds. Furthermore, the time derivative of the modified energy satisfies
\begin{equation}\label{modified_energy_dt}
\begin{aligned}
&\frac{d \tilde E[\phi,r]}{dt}  = \int_\Omega\frac{\delta \tilde E}{\delta \phi}\phi_t + \frac{\delta \tilde E}{\delta r}r'\phi_t d \mathbf{x}\\
&=\int_\Omega \left(\varepsilon^2 \Delta\phi +\frac{1}{2}(r+S)(1-2\phi)+\frac{1}{2}\left(\phi(1-\phi)-S \right)r'-\frac{1}{2}rr' \right) \phi_t d\mathbf{x} \\
&= \int_\Omega\left( \mathcal{G} ^{-1} \partial_t^\alpha\phi\right)\phi_t d \mathbf{x}.
    \end{aligned}
\end{equation}

We proceed to discretize the relaxation system (\ref{relaxation_ch}).
Integrating Eq.~(\ref{relaxation_ch}) over $[t_{n}, t_{n+1}]$, and using the $L1^{+}$ scheme for the Caputo term, we obtain the following linear relaxation $L1^+$-CN scheme
\begin{subequations}\label{Dis_RRER_CH}
    \begin{align}
        &\partial_\tau^\alpha \phi^{n+\frac{1}{2}} = M \Delta \mu^{n+\frac{1}{2}},\label{Dis_RRER_CH1}\\
        &\mu^{n+\frac{1}{2}} = -\varepsilon^2 \Delta \phi^{n+\frac{1}{2}} + \frac{1}{2} \left(r^{n+\frac{1}{2}} + S \right) \left(1-2\phi^{n+\frac{1}{2}} \right),\label{Dis_RRER_CH2}\\
        &\frac{r^{n+\frac{1}{2}} +r^{n-\frac{1}{2}} }{2} = \phi^n\left(1-\phi^{n }\right) - S.\label{Dis_RRER_CH3}
    \end{align}
\end{subequations}
The corresponding discrete modified energy is
\begin{align}\label{energydisch}
    &\tilde {E} \left [ \phi^{n+1},r^{n+\frac{1}{2}} \right ]\\
    &= \int_\Omega \frac{\varepsilon^2}{2}|\nabla \phi^{n+1}|^2 + \frac{1}{2} \left(r^{n+\frac{1}{2}} +S\right)\left( \phi^{n+1}\left(1-\phi^{n+1}\right) -1  - S \right) -\frac{1}{4}\left( r^{n+\frac{1}{2}} \right)^2 \ d \mathbf{x} + \frac{S^2}{4}|\Omega|.
\end{align}
Define $r^{\frac{1}{2}} =   \phi^0(1-\phi^0) - S$, and extend $r^{-\frac{1}{2}} = \phi^0(1-\phi^0) - S$. It is easy to verify that the modified energy satisfies $\tilde{E}\left[ \phi^0, r^{-\frac{1}{2}} \right] = E\left[ \phi^0 \right]$.

\begin{theorem}\label{theorem_energy_TFCH}
The $L1^+$-CN linear relaxation scheme (\ref{Dis_RRER_CH}) preserves the modified energy law:
    \begin{equation}\label{disenergylawch}
        \tilde E\left[ \phi^{n+1} , r^{n+\frac{1}{2}}\right] \leq   \tilde E\left[ \phi^{0} , r^{-\frac{1}{2}}\right]=  E\left[ \phi^0\right],\ \text{for}\ n\geq 0,
    \end{equation}
where $\tilde E\left[ \phi^{n+1} , r^{n+\frac{1}{2}}\right]$ defined in (\ref{energydisch}).
Under the time-step ratio constraint (\ref{time_step_constraint}), it preserves the discrete modified variational energy dissipation property that
\begin{equation}\label{modified_energy_decay}
\tilde{E}_\alpha \left[\phi^{n+1}, r^{n+ \frac{1}{2}}\right]-\tilde{E}_\alpha \left[\phi^{n}, r^{n- \frac{1}{2}}\right]=- \frac{1}{M } \mathcal{R}_{1-\alpha} \left( \nabla_\tau \phi^{n+1}\right)  ,\ \text{for}\ n \geq 2,
\end{equation}
with
\begin{equation}
\tilde{E}_\alpha \left[ \phi^{n+1}, r^{n+\frac{1}{2}} \right]:=\tilde{E}\left[ \phi^{n+1}, r^{n+\frac{1}{2}} \right] +\frac{1}{M  } \mathcal{A}_{1-\alpha}\left( \nabla_\tau \phi^{n+1}\right).
\end{equation}
\end{theorem}

\begin{proof}
    Taking $L^2$ inner product of Eq. \eqref{Dis_RRER_CH1} and Eq. \eqref{Dis_RRER_CH2}  with $   \Delta^{-1}\nabla_\tau \phi^{k+1}$, setting $\psi^k = \nabla\left(\Delta^{-1} \phi^k\right)$, and summing the resulting equations, we derive that
    \begin{equation} \label{proof1}
         \frac{1}{M}\left\langle  \partial_\tau^\alpha \psi^{k+ \frac{1}{2}}, \nabla_\tau  \psi^{k+1} \right\rangle =\left\langle  \mu^{k+\frac{1}{2}},  \nabla_\tau\phi^{k+1}\right\rangle.
    \end{equation}

We notice that the following relationship holds
\begin{equation} \label{proof2}
\begin{aligned}
&\left(  \frac{1}{2} r^{k+\frac{1}{2}} + \frac{S}{2}\right) \left(1-(\phi^{k+1} + \phi^k )\right) \nabla_\tau \phi^{k+1}  \\
&=\left(  \frac{1}{2} r^{k+\frac{1}{2}} + \frac{S}{2}\right) \nabla_\tau \phi^{k+1}
-  \left( \frac{1}{2} r^{k+\frac{1}{2}} + \frac{S}{2} \right)(\phi^{k+1} + \phi^k) \nabla_\tau \phi^{k+1}  \\
&= \frac{1}{2} \left(   r^{k+\frac{1}{2}} + S\right)  \nabla_\tau \phi^{k+1} + \frac{1}{2} \left ( r^{k-\frac{1}{2}} + S\right ) \left( \phi^k - \phi^k\right) + \frac{1}{2}  \left( r^{k-\frac{1}{2}} + S \right) \left( (\phi^k)^2 - (\phi^k)^2 \right) \\
&\quad - \frac{1}{2}  \left( r^{k+\frac{1}{2}} + S \right) \left( (\phi^{k+1})^2 - (\phi^k)^2 \right) \\
&= \frac{1}{2}  \left( r^{k+\frac{1}{2}} + S \right) \left( (\phi^{k+1})^2 - S \right)
- \frac{1}{2}   \left( r^{k-\frac{1}{2}} + S \right) \left( (\phi^k - \phi^k)^2 - S \right)  \\
&\quad - \frac{1}{4} \left( r^{k+\frac{1}{2}} \right)^2 +\frac{1}{4}  \left( r^{k-\frac{1}{2}} \right)^2
\end{aligned}
\end{equation}
Thus, it demonstrates that
    \begin{equation}\label{Dis_RRER_CH2_ener}
        \begin{aligned}
        &\left \langle\nabla_\tau \phi^{k+1}, \mu^{k+\frac{1}{2}}\right \rangle
        = \left\langle \varepsilon^2 \Delta \frac{\phi^k + \phi^{k+1}}{2} + \left( \frac{1}{2} r^{k+\frac{1}{2}} + \frac{S}{2} \right) \left( 1 - 2 \frac{\phi^k + \phi^{k+1}}{2} \right), \nabla_\tau \phi^{k+1} \right\rangle \\
        &= \frac{\varepsilon^2}{2} \left( \|\nabla \phi^{k+1}\|^2_{L^2(\Omega)} - \|\nabla \phi^k\|^2_{L^2(\Omega)} \right)
        + \left\langle \frac{1}{2} r^{k+\frac{1}{2}} + \frac{S}{2}, \nabla_\tau \phi^{k+1} \right\rangle  - \left\langle \left( \frac{1}{2} r^{k+\frac{1}{2}} + \frac{S}{2} \right) (\phi^{k+1} + \phi^k), \nabla_\tau \phi^{k+1} \right\rangle\\
        &= \frac{\varepsilon^2}{2} \left( \|\nabla \phi^{k+1}\|^2_{L^2(\Omega)} - \|\nabla \phi^k\|^2_{L^2(\Omega)} \right\rangle
        + \frac{1}{2} \left\langle (r^{k+\frac{1}{2}} + S)\left(\phi^{k+1} - (\phi^{k+1})^2 - S \right), 1 \right\rangle \\
        &\quad - \frac{1}{2} \left\langle (r^{k-\frac{1}{2}} + S)(\phi^k - (\phi^k)^2 - S), 1 \right\rangle
        - \frac{1}{4} \left\langle (r^{k+\frac{1}{2}})^2 - (r^{k-\frac{1}{2}})^2, 1 \right\rangle \\
        &= \tilde E\left[\phi^{k+1},r^{k+\frac{1}{2}}\right] - \tilde E\left[\phi^{k},r^{k-\frac{1}{2}}\right] = -\frac{1}{M}\left( \nabla_\tau\psi^{k+1}, \partial_\tau^\alpha \psi^{k+\frac{1}{2}} \right) .
        \end{aligned}
    \end{equation}

By summing $\tilde E\left[\phi^{k+1},r^{k+\frac{1}{2}}\right] - \tilde E\left[\phi^{k},r^{k-\frac{1}{2}}\right]$ from $k=0$ to $n$, the discrete energy boundedness law (\ref{disenergylawch}) is established for the $L1^+$-CN linear relaxation scheme of the time fractional Cahn-Hilliard equation. The modified variational energy dissipation law (\ref{modified_energy_decay}) is rigorously established by directly employing the discrete gradient structure (\ref{DGS}) from Lemma \ref{lemma_DGS}.

This completes the proof.
\end{proof}

\begin{remark}\label{remark 2.3}
By an argument analogous to the proof of Remark \ref{remark2.1}, the discrete modified energy (\ref{energydisch}) is a second-order accurate approximation in time to the original energy:
    \begin{equation}
    \begin{aligned}
        E\left[ \phi^{n+1} \right] &\approx \tilde{E} \left[ \phi^{n+1}, r^{n+1} \right]\approx \tilde E\left[\phi^{n+1}, r^{n+\frac{1}{2}}\right]+\frac{1}{4}\left\|r^{n+1}-r^{n+\frac{1}{2}}  \right\|_{L^2}^2.
    \end{aligned}
\end{equation}
\end{remark}

\subsection{Model III: Time fractional Swift-Hohenberg equation}
In this section, we will consider the time fractional Swift-Hohenberg equation
\begin{equation}\label{TFSH}
    \partial_t^\alpha \phi = -M \mu,\ \mu = \frac{\delta E}{\delta \phi}.
\end{equation}
The free energy is defined as
\begin{equation}
E[\phi]=\int_{\Omega} \frac{1}{2} \phi(1+\Delta)^2 \phi +F[\phi] \ d \mathbf{x} , \ \text { with }\ F(\phi)=  \frac{1}{4} \phi^4-\frac{g}{3} \phi^3+\frac{\delta}{2} \phi^2 .
\end{equation}
It can be rewritten as
\begin{equation}\label{SHF}
F(\phi) =  \left( \frac{1}{2} \phi^2 - \frac{g}{3} \phi +c_1 \right)^2 + c_2 \phi +c_3 ,
\end{equation}
where $c_1 =\frac{\delta}{2} - \frac{g^2}{9}, c_2 =\frac{g \delta}{3} - \frac{2 g^3}{27} $, and $c_3 = - \left( \frac{\delta}{2} - \frac{g^2}{9} \right)^2$.

We introduce an auxiliary variable
\begin{equation}
    r = \frac{1}{2} \phi^2 - \frac{g}{3} \phi +c_1-S.
\end{equation}
The nonlinear potential term is equivalent to
\begin{equation}
    F[\phi,r] = \int_\Omega (r + S)^2 +c_2\phi+ c_3 \ d\mathbf{x},
\end{equation}
and
\begin{equation}
    F'[\phi] = \int_\Omega 2(r+S) \cdot r'+ c_2 \ d\mathbf{x}, \ r' = \phi - \frac{g}{3}.
\end{equation}
The relaxation scheme for Eq.~(\ref{TFSH}) is reformulated as
\begin{subequations}\label{shrelaxation}
    \begin{align}
        &\partial_t^\alpha \phi  = - M \mu, \label{relx_SH1}\\
        &\mu = \left( 1+\Delta \right)^2 \phi + 2\left(r + S\right) \left(\phi - \frac{g}{3} \right)+c_2, \label{relx_SH2}\\
        &r = \frac{1}{2} \phi^2 - \frac{g}{3} \phi + c_1 - S ,\label{relx_SH3}
    \end{align}
\end{subequations}
with the modified free energy is defined by
\begin{equation}\label{modienergySH}
    \tilde {E} [\phi,r ] := \int_\Omega \frac{1}{2}\phi\left(1 + \Delta \right)^2 \phi + 2 ( r + S )\left(\frac{1}{2} \phi^2 - \frac{g}{3} \phi +c_1 - S \right) - r^2 +c_2 \phi \ d \mathbf{x}  +\left( c_3 +S^2\right) |\Omega|.
\end{equation}

The system (\ref{shrelaxation}) is discretized as follows
\begin{subequations}\label{Dis_SH}
    \begin{align}
        \partial_\tau^\alpha \phi^{n+\frac{1}{2}} & = -M \mu^{n+\frac{1}{2}}, \label{dis_sh1}\\
        \mu^{n+\frac{1}{2}} & = \left( 1+\Delta \right)^2 \phi^{n+\frac{1}{2}} +  2\left(r^{n+\frac{1}{2}} + S \right)\left( \phi^{n+\frac{1}{2}} - \frac{g}{3}\right)+c_2,\label{dis_sh2}\\
        \frac{r^{n+\frac{1}{2}} + r^{n-\frac{1}{2}}}{2} & = \frac{1}{2}\left(\phi^n\right)^2 - \frac{g}{3} \phi^n + c_1 - S.\label{dis_sh3}
    \end{align}
\end{subequations}
The discrete modified energy is defined as
\begin{align}\label{dismodienergySH}
    \tilde {E} [\phi^{n+1},r^{n+\frac{1}{2}} ] &:= \int_\Omega \frac{1}{2}\phi^{n+1}\left(1 + \Delta \right)^2 \phi^{n+1} + 2 ( r^{n+\frac{1}{2}} + S )\left(\frac{1}{2} \left(\phi^{n+1} \right)^2 - \frac{g}{3} \phi^{n+1} +c_1 - S \right) \\
    &\quad - \left(r^{n+\frac{1}{2}} \right)^2 +c_2 \phi^{n+1} \ d \mathbf{x}  +\left( c_3 +S^2\right) |\Omega|.
\end{align}

\begin{theorem}\label{theoremSH}
The $L1^+$-CN linear relaxation scheme (\ref{Dis_SH}) preserves the modified energy law:
    \begin{equation}\label{disenergylawsh}
        \tilde E\left[ \phi^{n+1} , r^{n+\frac{1}{2}}\right]\leq   \tilde E\left[ \phi^{0} , r^{-\frac{1}{2}}\right]=  E\left[ \phi^0\right],\ \text{for}\ n\geq 0,
    \end{equation}
where the discrete modified energy is defined in Eq.~(\ref{dismodienergySH}).

Under the time-step ratio constraint (\ref{time_step_constraint}), it preserves the modified variational energy dissipation property that
\begin{equation} \label{modified_energy_decaysh}
\tilde{E}_\alpha \left[\phi^{n+1}, r^{n+ \frac{1}{2}} \right]-\tilde{E}_\alpha \left[\phi^{n}, r^{n- \frac{1}{2}} \right]=- \frac{1}{M } \mathcal{R}_\alpha \left( \nabla_\tau \phi^{n+1}\right) \leq 0,\ \text{for}\ n \geq 2,
\end{equation}
with
\begin{equation}
\tilde{E}_\alpha \left[ \phi^{n+1}, r^{n+\frac{1}{2}},q^{n+\frac{1}{2}} \right]:=\tilde{E}\left[ \phi^{n+1}, r^{n+\frac{1}{2}},q^{n+\frac{1}{2}} \right] +\frac{1}{M  } \mathcal{A}_\alpha\left( \nabla_\tau \phi^{n+1}\right).
\end{equation}
\end{theorem}

\begin{proof}
Take the inner product of Eq.~(\ref{dis_sh1}) and Eq.~(\ref{dis_sh2}) with $\nabla_\tau \phi^{n+1}$, and sum them together. There is
\begin{equation}\label{SH_analy}
    \begin{aligned}
        &-\frac{1}{M}\left\langle\partial_t^\alpha \phi^{n+\frac{1}{2}}, \nabla_\tau \phi^{n+1} \right\rangle \\
        &= \left \langle\left( 1+\Delta \right)^2 \phi^{n+\frac{1}{2}}, \nabla_\tau \phi^{n+1}\right \rangle +  \left \langle 2\left(r^{n+\frac{1}{2}} + S \right) \left(\phi^{n+\frac{1}{2}}-\frac{g}{3} \right)+c_2, \nabla_\tau \phi^{n+1}\right \rangle.
    \end{aligned}
\end{equation}
 For the first term on the RHS of Eq.~(\ref{SH_analy}), we have
\begin{equation}\label{interf_term1}
    \begin{aligned}
        \left \langle \left( 1+\Delta \right)^2  \phi^{n+\frac{1}{2}}, \nabla_\tau \phi^{n+1}\right \rangle &= \left \langle \left( 1+\Delta \right)  \phi^{n+\frac{1}{2}},\nabla_\tau  \left( 1+\Delta \right) \phi^{n+1}\right \rangle\\
        & = \frac{1}{2} \left( \left \| \left( 1+\Delta \right) \phi^{n+1} \right\|^2 - \left \| \left( 1+\Delta \right) \phi^{n } \right\|^2 + \left \| \nabla_\tau\left( 1+\Delta \right) \phi^{n+1} \right\|^2\right).
    \end{aligned}
\end{equation}
For the second term of RHS, we have
\begin{equation}\label{interf_term3}
    \begin{aligned}
         &\left \langle 2\left(r^{n+\frac{1}{2}} + S \right) \left(\phi^{n+\frac{1}{2}}-\frac{g}{3} \right)+c_2, \nabla_\tau \phi^{n+1}\right \rangle \\
         &= \left \langle   2\left( r^{n+ \frac{1}{2}} +S \right) \left( \frac{\phi^{n+1}+ \phi^n}{2} - \frac{g}{3}\right)  , \nabla_\tau \phi^{n+1}  \right \rangle + \left \langle  c_2\phi^{n+1} - c_2\phi^n,1 \right \rangle\\
         & =  \left \langle  2  \left( r^{n+ \frac{1}{2}} +S \right)  \left(\left( \frac{1}{2} \left(\phi^{n+1}\right)^2  -\frac{g}{3}\phi^{n+1} +c_1 - S \right)- \left(\frac{1}{2} \left(\phi^{n }\right)^2-\frac{g}{3}\phi^n+c_1 - S\right)    \right) ,1 \right \rangle\\
         &\quad +  \left \langle  2  \left( r^{n- \frac{1}{2}} +S \right) \left( \left(\frac{1}{2} \left(\phi^{n }\right)^2-\frac{g}{3}\phi^n+c_1 - S\right) -\left(\frac{1}{2} \left(\phi^{n }\right)^2-\frac{g}{3}\phi^n+c_1 - S\right)  \right) ,1 \right \rangle\\
         &\quad + \left \langle  c_2\phi^{n+1} - c_2\phi^n,1 \right \rangle\\
         & =  \left \langle 2   \left( r^{n+ \frac{1}{2}} +S \right)  \left( \frac{1}{2} \left(\phi^{n+1}\right)^2  -\frac{g}{3}\phi^{n+1} +c_1 - S \right)-\left(r^{n+\frac{1}{2}} \right)^2 +c_2 \phi^{n+1},1 \right \rangle  \\
         &-\left \langle   2 \left( r^{n- \frac{1}{2}} +S \right)  \left( \frac{1}{2} \left(\phi^{n }\right)^2  -\frac{g}{3}\phi^{n } +c_1 - S \right)- \left(r^{n- \frac{1}{2}}   \right) ^2 +c_2 \phi^n ,1 \right \rangle  .
    \end{aligned}
\end{equation}

By using Eq.~(\ref{SH_analy}), Eq.~(\ref{interf_term1}), and Eq.~(\ref{interf_term3}), we derive that
\begin{equation}
    \tilde{E}\left[\phi^{n+1}, r^{n+\frac{1}{2}} \right] - \tilde{E}\left[\phi^{n }, r^{n-\frac{1}{2}} \right] = -\frac{1}{M}\left\langle\partial_t^\alpha \phi^{n+\frac{1}{2}}, \nabla_\tau \phi^{n+1} \right\rangle.
\end{equation}
Thus, the proof can be completed by following the same argument used at the end of the proof of Theorem \ref{theorem_energy_TFAC}.
\end{proof}

\begin{remark}\label{remark 2.5}
This remark clarifies the relationship between the modified discrete energy and its original counterpart.
    \begin{equation}
        \begin{aligned}\label{remarkac}
         &E\left[\phi^{n+1} \right] \approx \tilde E\left[ \phi^{n+1}, r^{n+1} \right]\\
         &=  \left \langle \frac{1}{2}\phi^{n+1}\left(1 + \Delta \right)^2 \phi^{n+1} +  2 \left( r^{n+1} + S \right)\left(\frac{1}{2} \left(\phi^{n+1} \right)^2 - \frac{g}{3} \phi^{n+1} +c_1 - S \right)- \left(r^{n+1} \right)^2  ,1 \right \rangle \\
         & \quad + \left \langle c_2 \phi^{n+1} , 1\right \rangle + \left( c_3 + S^2 \right)|\Omega|\\
         &=   \left \langle \frac{1}{2}\phi^{n+1}\left(1 + \Delta \right)^2 \phi^{n+1} , 1\right \rangle + \left \langle 2 \left( r^{n+\frac{1}{2}} + S \right)\left(\frac{1}{2} \left(\phi^{n+1} \right)^2 - \frac{g}{3} \phi^{n+1} +c_1 - S \right)- \left(r^{n+\frac{1}{2}} \right)^2,1 \right \rangle \\
         &\quad + \left \langle c_2 \phi^{n+1} , 1\right \rangle + \left( c_3 + S^2 \right)|\Omega|+ \left \langle 2 \left( r^{n+1} -r^{n+\frac{1}{2}} \right)\left(\frac{1}{2} \left(\phi^{n+1} \right)^2 - \frac{g}{3} \phi^{n+1} +c_1 - S \right) ,1 \right \rangle \\
         & - \left \langle \left( r^{n+1}\right)^2- \left(r^{n+ \frac{1}{2}} \right)^2,1 \right \rangle\\
        & = \tilde {E}\left[\phi^{n+1}, r^{n+\frac{1}{2}}  \right]+ \left \|    r^{n+1} -r^{n+\frac{1}{2}}   \right \|^2.
    \end{aligned}
    \end{equation}
The above formulation shows that the modified discrete energy is a second-order approximation of its original counterpart.
\end{remark}

\begin{remark}\label{remark 2.6}
(Asymptotic compatibility)
Using the analysis framework established in Remark \ref{remark 2.2}, we can show that the linear relaxation scheme (\ref{Dis_SH}) for the time-fractional Swift-Hohenberg equation reduces to the CN scheme of the integral counterpart as $\alpha  \rightarrow  1^-$. Its variational energy dissipation law (\ref{modified_energy_decaysh}) is asymptotically compatible with their integer-order counterparts that
\begin{equation}\label{acintegralenergy}
    \partial_\tau \tilde E\left[\phi^n , r^{n-\frac{1}{2}} \right] = -M^{-1}\left\| \partial_\tau \phi^n \right\|^2  = -M \left\| \mu^{n-\frac{1}{2}}   \right\|^2, \ 1 \leq n \leq N.
\end{equation}

\end{remark}

\section{Numerical results}
\label{section4}
This section presents several numerical examples to illustrate the effectiveness of the proposed numerical method. We consider the periodic boundary condition and employ the Fourier spectral method with a $128 \times 128$ mesh for spatial discretization. Without specific needs, the stability parameter is fixed as $S=2$.

\begin{example}\label{example1}
(Convergence tests)
We first present an example in which an appropriate source term is selected so that the time-fractional volume-conserved Allen-Cahn equation, Cahn-Hilliard equation, and Swift-Hohenberg equation admit the following exact solution:
\begin{equation}
    \phi(x,y,t) = \left(1-\omega_{1+\sigma}(t)\right)\left( \frac{1}{4}\sin(2x)\cos(2y) + 0.45\right),\ t\in[0,T].
\end{equation}
\end{example}
We use graded temporal meshes $t_n = \left(\frac{n}{N}\right)^{\gamma}T$ for different values of $\alpha$. The mobility coefficient $M=0.01$. Numerical results are listed from Table \ref{tab:CACsigma06} to Table \ref{tab:SH}. For the time-fractional volume-conserved Allen-Cahn and Cahn-Hilliard models, we set the interface width $\varepsilon=0.25$. For the time-fractional Swift-Hohenberg equation, the parameters are chosen as $ g=1$ and $ \delta=0.2.$

Tables \ref{tab:CACsigma06} and \ref{tab:CACsigma03} show the convergence rates of the time-fractional volume-conserved Allen-Cahn equation with different $\alpha$ and $\sigma$ on graded grids with different values of the graded parameter $\gamma$. The numerical results demonstrate that scheme (\ref{Dis_RRER_AC}) achieves a temporal convergence order of $O(N^{-\min\{\gamma \sigma,2\}})$ for both $\phi$ and the auxiliary variable $r$. This convergence rate is independent of $\alpha$ and aligns with the findings reported in \cite{ji2020adaptive}.
Specifically, when $\gamma < \frac{2}{\sigma}$, the temporal convergence order becomes $O(N^{-\gamma \sigma})$. To attain the optimal second-order temporal convergence, the graded parameter should satisfy $\gamma \geq \frac{2}{\sigma}$.

Tables \ref{tab:CH} and \ref{tab:SH} show the convergence rates resulting from tests of the time-fractional Cahn–Hilliard and time-fractional Swift–Hohenberg equations, respectively, with the optimal mesh partition ($\gamma=2/\sigma$). It can be seen that both of the variables computed in schemes (\ref{Dis_RRER_CH}) and (\ref{Dis_SH})  linearly achieve second-order convergence for different fractional-order parameter $\alpha$. This confirms the convergence behavior of $O(N)^{-\min\{\gamma \sigma,2\}}$ observed in the numerical tests presented in Tables \ref{tab:CACsigma06} and \ref{tab:CACsigma03}.

\begin{table}[H]
\centering
\caption{$L^\infty$-errors and convergence rates of $\phi$ and $r$ for the relaxation scheme (\ref{Dis_RRER_AC}) of the time-fractional volume-conserved Allen-Cahn equation with $(\alpha,\sigma)=(0.4,0.6)$ and $T = 1$.}
\begin{tabular}{c c c c c c c c}
\toprule
 & & \multicolumn{2}{c}{$\gamma=2$} & \multicolumn{2}{c}{$\gamma=2/\sigma$} & \multicolumn{2}{c}{$\gamma=7$} \\
\cmidrule(lr){3-4} \cmidrule(lr){5-6} \cmidrule(lr){7-8}
 & $N$ & Error & Order & Error & Order & Error & Order \\
\midrule
\multirow{4}{*}{$\phi$} & 8   & $1.69E{-2}$ & -- & $3.22E{-3}$ & -- & $3.37E{-3}$ & -- \\
 & 16  & $1.11E{-2}$ & 0.60 & $8.23E{-4}$ & 1.97 & $8.55E{-4}$ & 1.98 \\
 & 32  & $7.35E{-3}$ & 0.60 & $2.07E{-4}$ & 1.99 & $2.11E{-4}$ & 2.02 \\
 & 64  & $4.85E{-3}$ & 0.60 & $5.17E{-5}$ & 2.00 & $5.43E{-5}$ & 1.96 \\
\midrule
\multirow{4}{*}{$r$} & 8   & $1.35E{-1}$ & -- & $8.59E{-3}$ & -- & $8.23E{-3}$ & -- \\
 & 16  & $9.28E{-2}$ & 0.54 & $2.43E{-3}$ & 1.82 & $2.30E{-3}$ & 1.84 \\
 & 32  & $6.28E{-2}$ & 0.56 & $6.52E{-4}$ & 1.90 & $6.00E{-4}$ & 1.94 \\
 & 64  & $4.21E{-2}$ & 0.58 & $1.69E{-4}$ & 1.95 & $1.51E{-4}$ & 1.99 \\
\bottomrule
\end{tabular}
\label{tab:CACsigma06}
\end{table}

\begin{table}[H]
\centering
\caption{$L^\infty$-errors and convergence rates of $\phi$ and $r$ for the relaxation scheme (\ref{Dis_RRER_AC}) of the time-fractional volume-conserved Allen-Cahn equation with $(\alpha,\sigma)=(0.7,0.3)$ and $T = 1$.}
\begin{tabular}{c c c c c c c c}
\toprule
 & & \multicolumn{2}{c}{$\gamma=1$} & \multicolumn{2}{c}{$\gamma=2/\sigma$} & \multicolumn{2}{c}{$\gamma=4$} \\
\cmidrule(lr){3-4} \cmidrule(lr){5-6} \cmidrule(lr){7-8}
 & $N$ & Error & Order & Error & Order & Error & Order \\
\midrule
\multirow{4}{*}{$\phi$} & 8   & $2.67E{-2}$ & -- & $8.36E{-4}$ & -- & $1.17E{-3}$ & -- \\
 & 16  & $1.75E{-2}$ & 0.60 & $2.12E{-4}$ & 1.98 & $2.96E{-4}$ & 1.98 \\
 & 32  & $1.15E{-2}$ & 0.60 & $5.34E{-5}$ & 1.99 & $7.47E{-5}$ & 1.99 \\
 & 64  & $7.60E{-3}$ & 0.60 & $1.34E{-5}$ & 1.99 & $1.88E{-5}$ & 1.99 \\
\midrule
\multirow{4}{*}{$r$} & 8   & $2.08E{-1}$ & -- & $2.44E{-3}$ & -- & $2.11E{-3}$ & -- \\
 & 16  & $1.47E{-1}$ & 0.50 & $6.27E{-4}$ & 1.96 & $5.05E{-4}$ & 2.06 \\
 & 32  & $1.03E{-1}$ & 0.52 & $1.58E{-4}$ & 1.99 & $1.32E{-4}$ & 1.93 \\
 & 64  & $7.18E{-2}$ & 0.52 & $3.95E{-5}$ & 2.00 & $3.49E{-5}$ & 1.92 \\
\bottomrule
\end{tabular}
\label{tab:CACsigma03}
\end{table}

\begin{table}[H]
\centering
\caption{$L^\infty$-errors and convergence rates of $\phi$ and $r$ for the relaxation scheme (\ref{Dis_RRER_CH}) of the time-fractional Cahn-Hilliard equation with $\sigma =2$, $\gamma =2/\sigma$, and $T = 1$.}
\begin{tabular}{c c c c c c c c c c}
\toprule
 & & \multicolumn{2}{c}{$\alpha = 0.3$} & \multicolumn{2}{c}{$\alpha = 0.6$} & \multicolumn{2}{c}{$\alpha = 0.9$} & \multicolumn{2}{c}{$\alpha = 1$} \\
\cmidrule(lr){3-4} \cmidrule(lr){5-6} \cmidrule(lr){7-8} \cmidrule(lr){9-10}
 & $N$ & Error & Order & Error & Order & Error & Order & Error & Order \\
\midrule
\multirow{4}{*}{$\phi$} & 8   & $2.09E{-2}$ & -- & $1.65E{-2}$ & -- & $8.17E{-3}$ & -- & $3.97E{-3}$ & -- \\
 & 16  & $5.44E{-3}$ & 1.94 & $4.48E{-3}$ & 1.88 & $2.28E{-3}$ & 1.84 & $1.01E{-3}$ & 1.98 \\
 & 32  & $1.39E{-3}$ & 1.97 & $1.18E{-3}$ & 1.93 & $6.19E{-4}$ & 1.88 & $2.52E{-4}$ & 2.00 \\
 & 64  & $3.48E{-4}$ & 1.99 & $3.04E{-4}$ & 1.95 & $1.66E{-4}$ & 1.90 & $6.31E{-5}$ & 2.00 \\
\midrule
\multirow{4}{*}{$r$} & 8   & $2.33E{-2}$ & -- & $2.13E{-2}$ & -- & $1.75E{-2}$ & -- & $1.57E{-2}$ & -- \\
 & 16  & $5.87E{-3}$ & 1.99 & $5.39E{-3}$ & 1.99 & $4.28E{-3}$ & 2.04 & $3.64E{-3}$ & 2.11 \\
 & 32  & $1.45E{-3}$ & 2.01 & $1.36E{-3}$ & 1.99 & $1.09E{-3}$ & 1.98 & $9.09E{-4}$ & 2.00 \\
 & 64  & $3.66E{-4}$ & 1.99 & $3.43E{-4}$ & 1.98 & $2.76E{-4}$ & 1.98 & $2.69E{-4}$ & 1.76 \\
\bottomrule
\end{tabular}
\label{tab:CH}
\end{table}

\begin{table}[H]
\centering
\caption{$L^\infty$-errors and convergence rates of $\phi$, $r$, and $q$ for the relaxation scheme of the time-fractional Swift-Hohenberg equation with $\sigma =2$, $\gamma =2/\sigma$, and $T=1$.}
\begin{tabular}{c c c c c c c c c c}
\toprule
 & & \multicolumn{2}{c}{$\alpha = 0.3$} & \multicolumn{2}{c}{$\alpha = 0.6$} & \multicolumn{2}{c}{$\alpha = 0.9$} & \multicolumn{2}{c}{$\alpha = 1$} \\
\cmidrule(lr){3-4} \cmidrule(lr){5-6} \cmidrule(lr){7-8} \cmidrule(lr){9-10}
 & $N_t$ & Error & Order & Error & Order & Error & Order & Error & Order \\
\midrule
\multirow{4}{*}{$\phi$} & 8   & 3.46E{-3} & --    & 2.21E{-3} & --    & 1.06E{-3} & --    & 2.82E{-4} & --    \\
 & 16  & 8.64E{-4} & 2.00  & 5.47E{-4} & 2.01  & 2.85E{-4} & 1.90  & 7.05E{-5} & 2.00  \\
 & 32  & 2.15E{-4} & 2.00  & 1.36E{-4} & 2.01  & 7.56E{-5} & 1.91  & 1.76E{-5} & 2.00  \\
 & 64  & 5.37E{-5} & 2.00  & 3.38E{-5} & 2.01  & 1.99E{-5} & 1.92  & 4.41E{-6} & 2.00  \\
\midrule
\multirow{4}{*}{$r$} & 8   & 6.10E{-3} & --    & 6.08E{-3} & --    & 5.95E{-3} & --    & 5.86E{-3} & --    \\
 & 16  & 1.91E{-3} & 1.68  & 1.90E{-3} & 1.67  & 1.87E{-3} & 1.67  & 1.83E{-3} & 1.68  \\
 & 32  & 5.29E{-4} & 1.85  & 5.28E{-4} & 1.85  & 5.17E{-4} & 1.85  & 5.07E{-4} & 1.85  \\
 & 64  & 1.39E{-4} & 1.93  & 1.39E{-4} & 1.93  & 1.36E{-4} & 1.93  & 1.33E{-4} & 1.93  \\
\bottomrule
\end{tabular}
\label{tab:SH}
\end{table}

\begin{example}\label{example2}
In this example, we consider the time-fractional volume-conserved Allen-Cahn and Cahn-Hilliard equations.
The adaptive time-stepping algorithm is used to improve computational efficiency. Given a coefficient $\lambda$, the minimum time $\tau_{\min}=10^{-3}$ and the maximum time step $\tau_{\max}=0.5$, we can update the time step size by the formula
\begin{equation}\label{adaptive_time_size}
    \tau_{n+1} = \max\left\{ \tilde{\tau}_n, H_{1-\alpha}\left(\rho_n \right)\tau_n \right\}, \ \tilde{\tau}_n = \max \left\{ \tau_{\min}, \frac{\tau_{\max}}{\sqrt{1+ \lambda \left\| \partial_\tau \phi^n\right\|^2}}\right\},
\end{equation}
where $H_{\alpha}(\cdot)$ defined in (\ref{time_step_constraint}). This adaptive approach has recently been used in other phase-field models \cite{qi2024unified}. We set a random initial condition distributed in the field of $(-0.2, 0.2)$ on the spatial domain $[0,2\pi]^2$ to test the two models. The parameters are chosen as $\lambda = 10^2$, $M=0.01$, and $\varepsilon = 0.25$ and $0.5$ for the time-fractional volume-conserved Allen-Cahn equation and Cahn-Hilliard equations, respectively.
\end{example}

Figure \ref{fig:CAC} shows the numerical results of the scheme (\ref{Dis_RRER_AC}) for the time-fractional volume-conserved Allen-Cahn equation over time $t=500$ with $\alpha = 0.4, 0.7, 1$. We can see from Figure (\ref{fig:CAC_Orig_Energy}) that for different $\alpha$, the original discrete energy $E$ and the modified discrete energy $\tilde E$ both decay rapidly at initial times and then gradually reach a steady state, with the energy decreasing monotonically throughout the entire process. This indicates that the developed scheme is energy stable. Furthermore, we note that the modified energy is highly consistent with the original energy, which verifies the statement of remark \ref{remark2.1}.
Figure (\ref{fig:CAC_Modi_Energy}) compares the modified discrete energy and the variational discrete energy, indicating that the variational energy tends to the modified energy as $\alpha\rightarrow 1$. In particular, the coarsening time increases as the fractional-order parameter $\alpha$ decreases, which is further verified in Figure \ref{fig:CAC_Snapshots}. Figure (\ref{fig:CAC_Mass_Conservation}) shows that the computed mass error fluctuates within a very small range, which can illustrate that the proposed scheme satisfies mass conservation. Figure (\ref{fig:CAC_Adaptive_Time_Size}) illustrates the evolution of time-step sizes selected by the adaptive strategy, which chooses smaller steps during significant energy changes and larger steps during minimal changes, effectively capturing the time-multiscale behavior. Furthermore, we investigate the $L^2$-error between the auxiliary variable $r$ and its original counterpart $\phi^2-1-S$. As shown in Figure (\ref{fig:CAC_r_eror}), the error is related to the phase transition rate. The error increases when the phase changes abruptly, and decreases as it gradually approaches a steady state. However, for the IEQ and SAV methods, since the auxiliary variable $r$ is obtained by solving a differential equation, the numerical errors between the original and auxiliary variables or between the original energy and the modified energy accumulate over time (see \cite{jiang2022improving}), which is well known in the community. We note that the proposed linear relaxation scheme introduces an algebraic equation to solve the auxiliary variable, so that the numerical inconsistency between the original and auxiliary variables does not deteriorate over time. This provides a new perspective for designing structure-preserving numerical schemes for phase-field models in the future.

Figures (\ref{fig:CH_Orig_Energy}) and (\ref{fig:CH_Modi_Energy}) illustrate the temporal evolution of the original energy, the modified energy, and the variational energy for the scheme (\ref{Dis_RRER_CH}) applied to the Cahn-Hilliard equation, which are consistent with the energy evolution behavior observed in Figure \ref{fig:CAC}. As demonstrated in Figure (\ref{fig:CH_r_eror}), the error of the auxiliary variable $r$ does not accumulate over time.

\begin{figure}[H]
  \centering
  \begin{subfigure}[b]{0.40\textwidth}
    \centering
    \includegraphics[width=\textwidth]{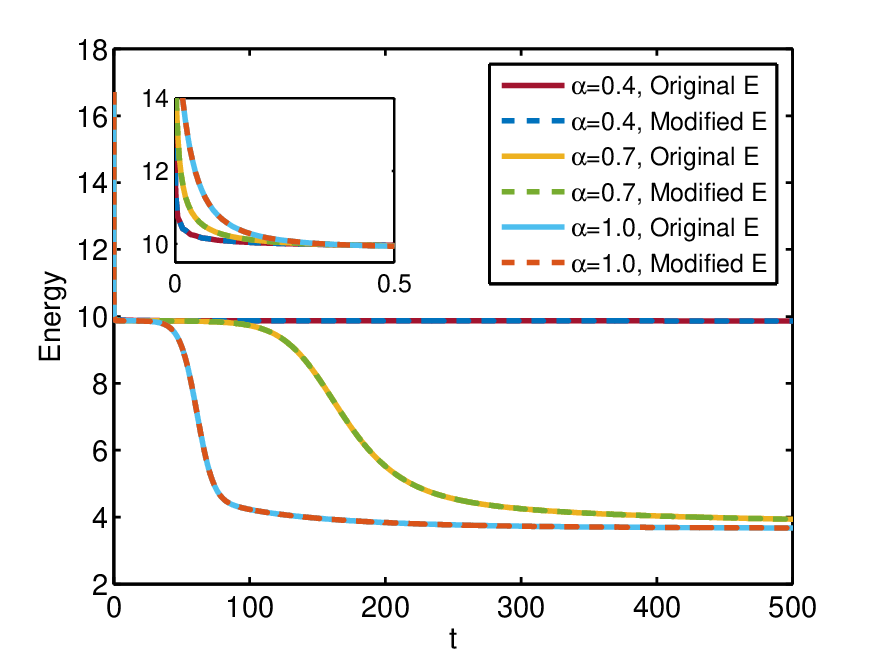}
    \caption{}
    \label{fig:CAC_Orig_Energy}
  \end{subfigure}
  \hfill
    \begin{subfigure}[b]{0.40\textwidth}
      \centering
      \includegraphics[width=\textwidth]{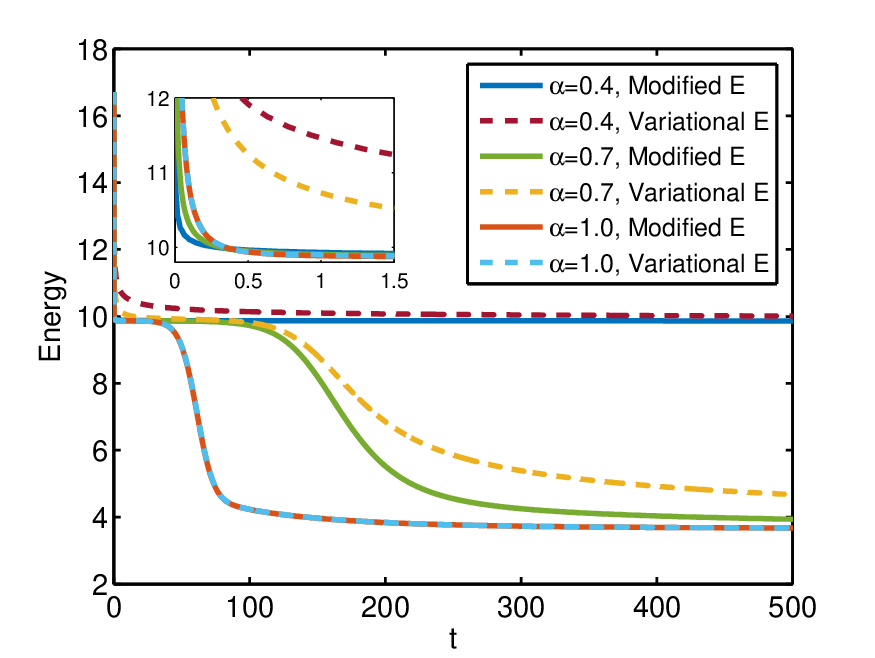}
      \caption{}
      \label{fig:CAC_Modi_Energy}
    \end{subfigure}  \hfill
    \begin{subfigure}[b]{0.33\textwidth}
      \centering
      \includegraphics[width=\textwidth]{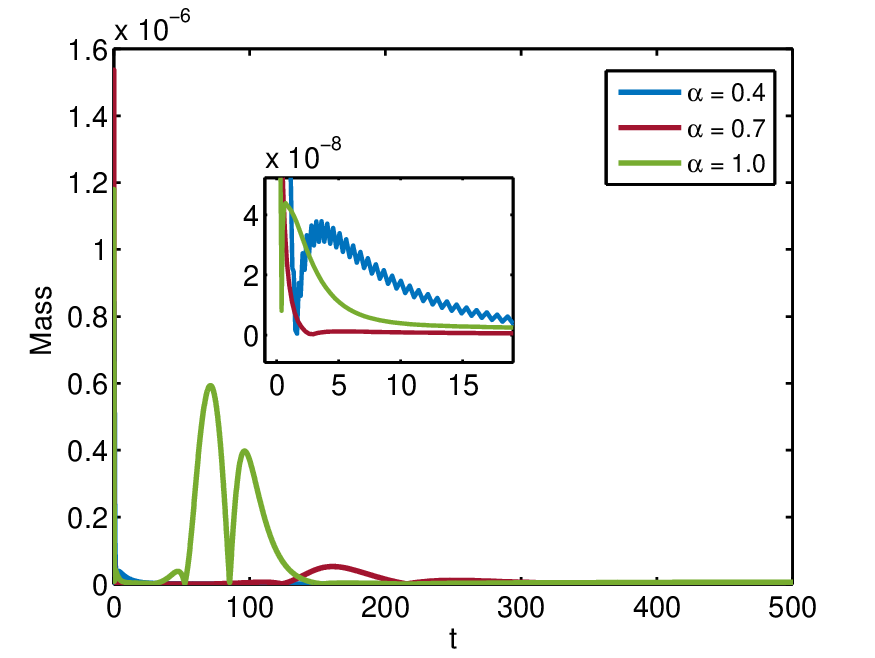}
      \caption{}
      \label{fig:CAC_Mass_Conservation}
    \end{subfigure}
    \hfill
    \begin{subfigure}[b]{0.33\textwidth}
      \centering
      \includegraphics[width=\textwidth]{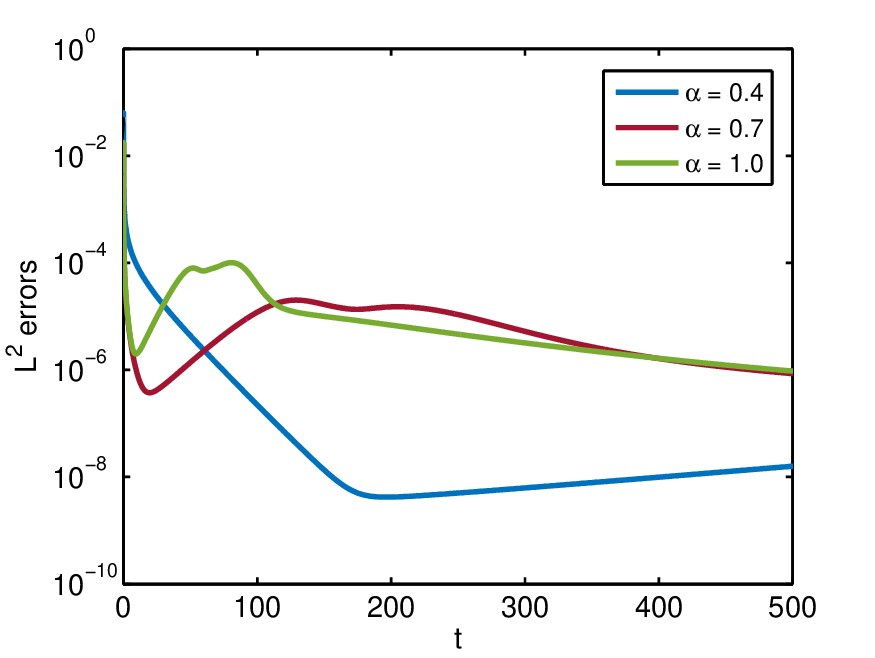}
      \caption{}
      \label{fig:CAC_r_eror}
    \end{subfigure}\hfill
    \begin{subfigure}[b]{0.33\textwidth}
      \centering
      \includegraphics[width=\textwidth]{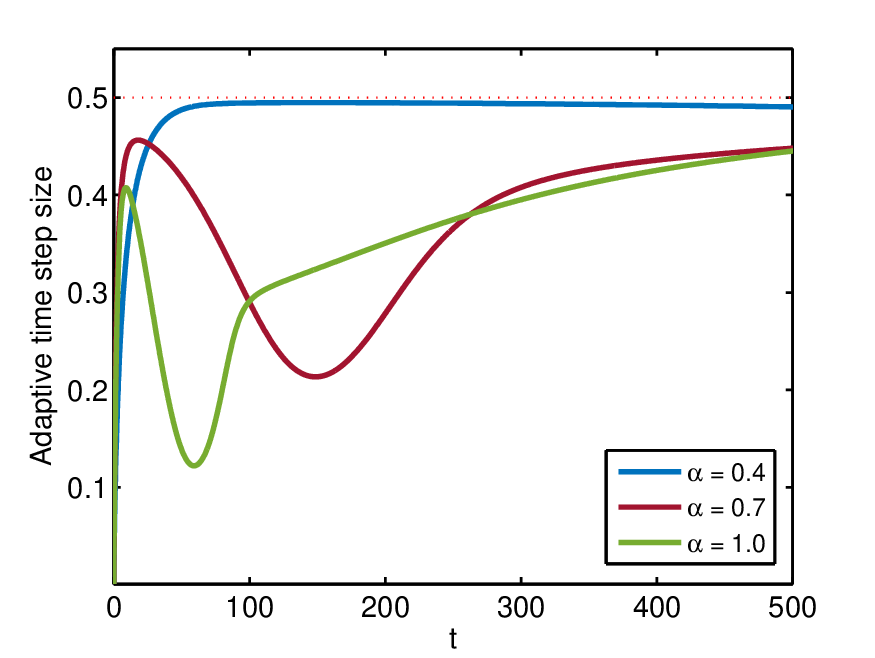}
      \caption{}
      \label{fig:CAC_Adaptive_Time_Size}
    \end{subfigure}
  \caption{The time-fractional volume-conserved Allen-Cahn equation. (a) The original energy and the modified energy. (b) The modified energy and the variational energy. (c) Mass results of $\int_\Omega (\phi^n -\phi^0) d \mathbf{x}$. (d) The $L^2$-error between $r^{n+1/2}$ and $\phi^{n+1/2} - 1 - S$. (e) Adaptive time-step size.
}
  \label{fig:CAC}
\end{figure}

\begin{figure}[H]
\centering
\includegraphics[width=1\textwidth ]{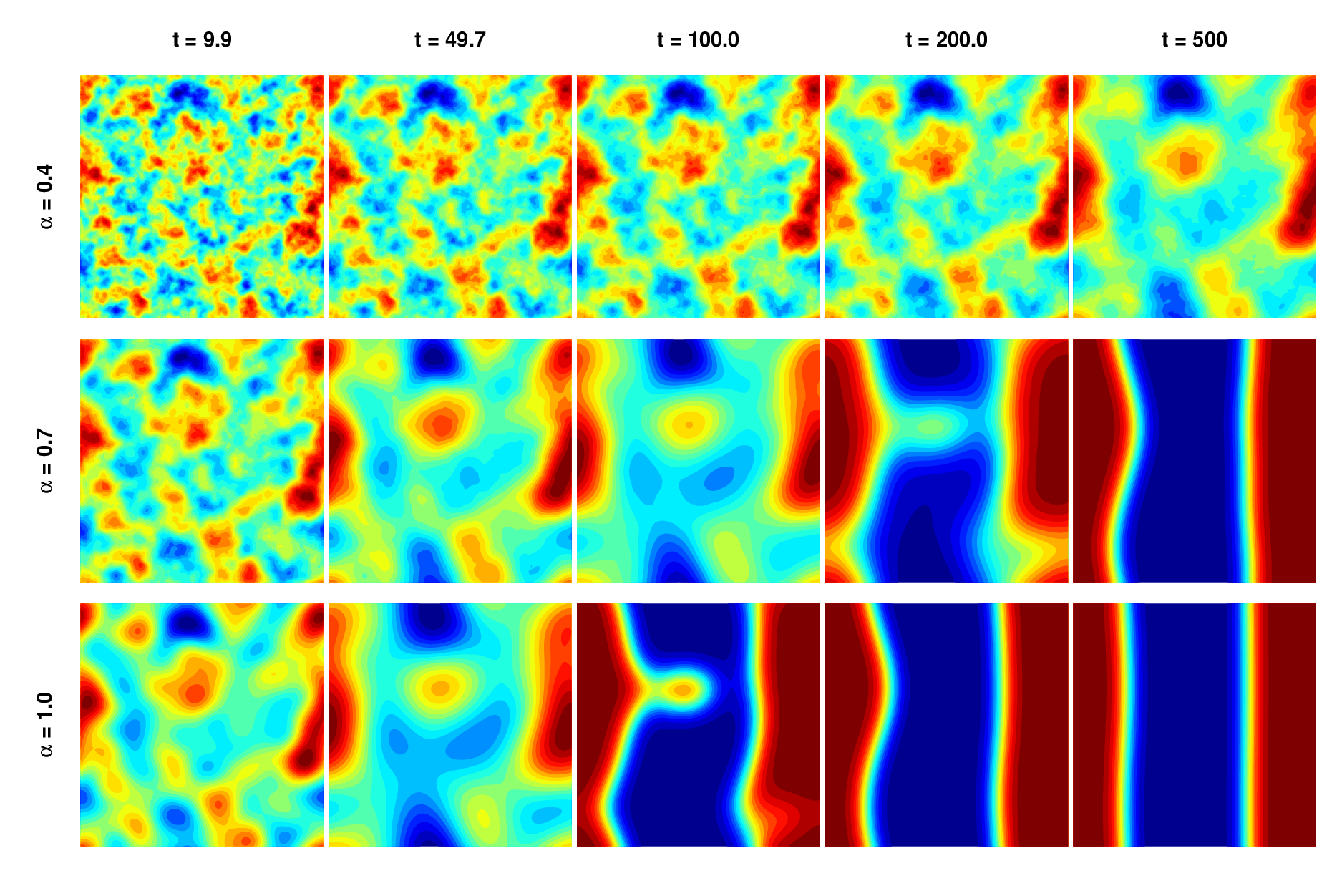}
\caption{The snapshots of the time-fractional volume-conserved Allen-Cahn equation with different $\alpha$.}
\label{fig:CAC_Snapshots}
\end{figure}

\begin{figure}[htpb]
  \centering
   \begin{subfigure}[b]{0.32\textwidth}
    \centering
    \includegraphics[width=\textwidth]{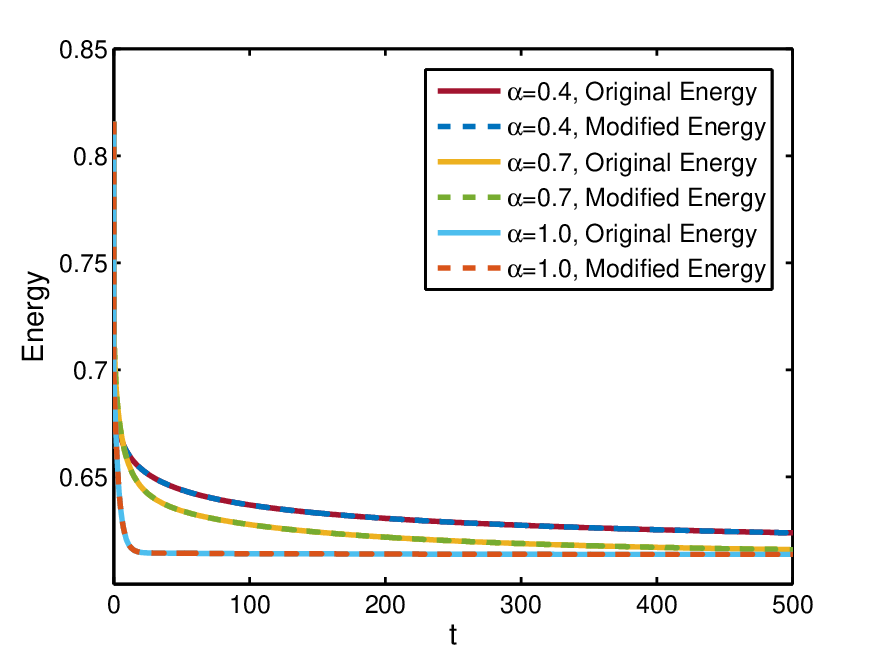}
    \caption{}
    \label{fig:CH_Orig_Energy}
  \end{subfigure}
    \hfill
    \begin{subfigure}[b]{0.32\textwidth}
      \centering
      \includegraphics[width=\textwidth]{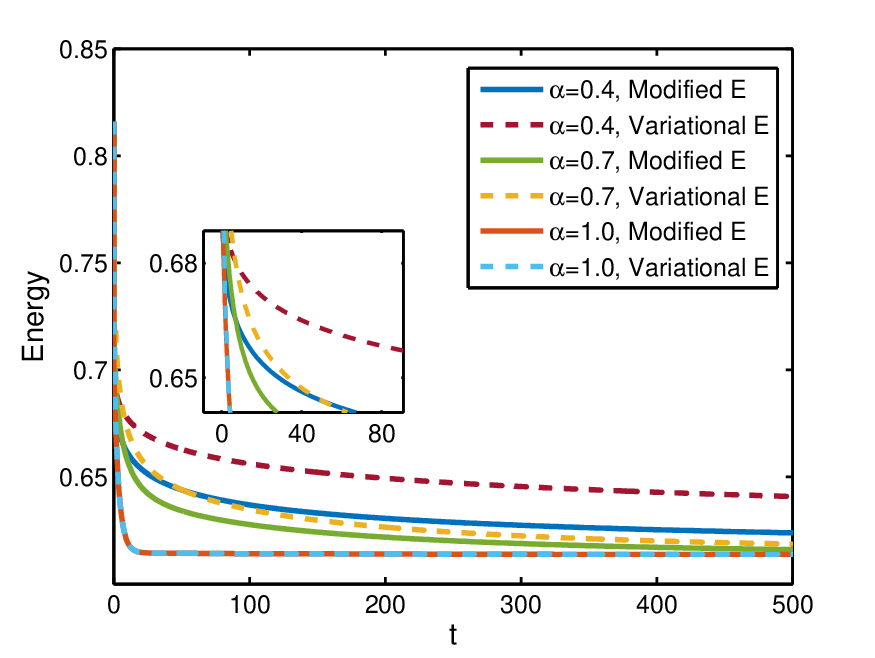}
      \caption{}
      \label{fig:CH_Modi_Energy}
    \end{subfigure}
        \hfill
    \begin{subfigure}[b]{0.32\textwidth}
      \centering
      \includegraphics[width=\textwidth]{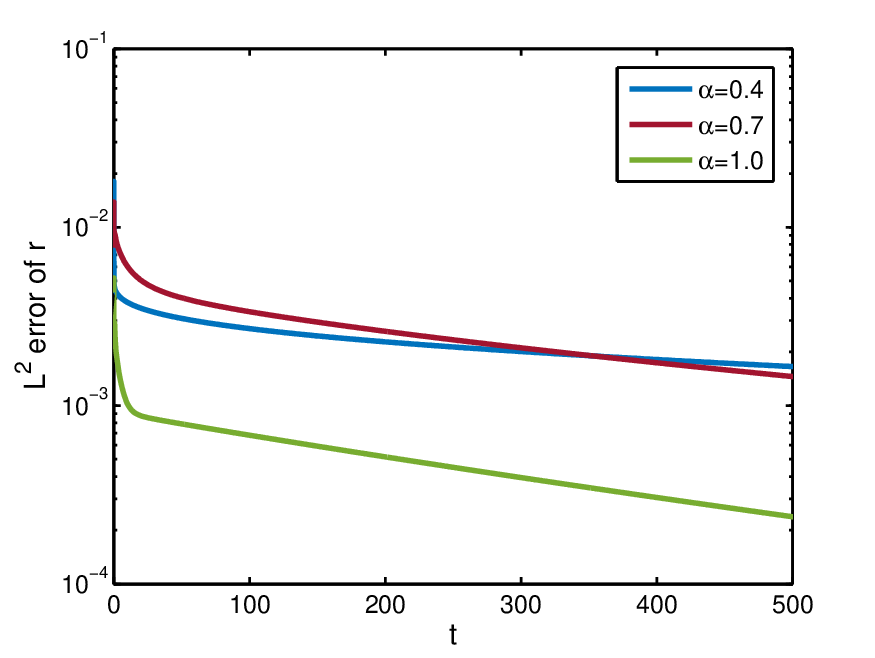}
      \caption{}
      \label{fig:CH_r_eror}
    \end{subfigure}
  \caption{The time-fractional Cahn-Hilliard equation. (a) The original discrete energy and the modified energy. (b) The modified discrete energy and the variation modified energy. (c) The $L^2$-error between $r^{n+1/2}$ and $\phi^{n+1/2}(1-\phi^{n+1/2})-S$. }
  \label{fig:CH}
\end{figure}

\begin{example}\label{example3}
In this example, we consider the time-fractional Swift-Hohenberg equation on the spatial domain $(0,32)^2$ with the following initial condition
    \begin{equation}
        \begin{aligned}
            \phi_0 &= 0.07  - 0.02 \cos\left( \frac{x-12}{32} 2\pi\right) \sin\left(  \frac{y-1}{32} 2\pi\right) \\
            &\quad + 0.02 \cos\left(  \frac{x+10}{32} \pi\right)^2 \sin\left( \frac{y+3}{32}\pi \right)^2   - 0.01 \sin^2\left(\frac{x}{32}4\pi \right) \sin^2\left(\frac{ y-6}{32}4\pi\right) .
        \end{aligned}
    \end{equation}
\end{example}

The parameters are chosen as $M = 0.6, \varepsilon=0.25, g=0.5$, $\delta = -0.25$. We use the adaptive time-stepping strategy with $\tau_{max} = 1, \tau_{\min} = 10^{-5}$, and $\lambda = 10^3$ over time $t=1000$ for $\alpha =0.4, 0.7, 0.9, 1$. Figure \ref{fig:SH} and \ref{fig:SH_Snapshots} plot the evolution of energy and the configurations at several observation times with different $\alpha$, respectively. The results presented are similar to those in Example 2. The fractional-order parameter $\alpha$ significantly affects the coarsening rate of the model. It is worth noting that the error between the auxiliary variable and the original variable shown in Figure (\ref{fig:SH_L2errors_r}) avoids accumulation over time, which further verifies the proposed scheme is consistent and suitable for long-term simulations.

\begin{figure}[H]
  \centering
   \begin{subfigure}[b]{0.48\textwidth}
    \centering
    \includegraphics[width=\textwidth]{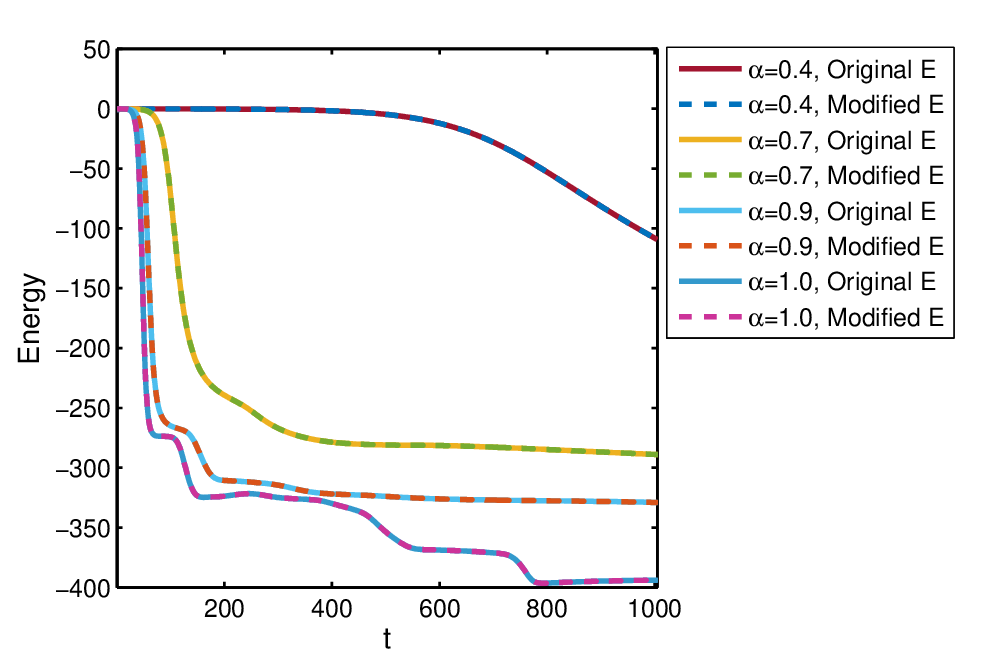}
    \caption{}
    \label{fig:SH_Orig_Energy}
  \end{subfigure}
    \hfill
    \begin{subfigure}[b]{0.48\textwidth}
      \centering
      \includegraphics[width=\textwidth]{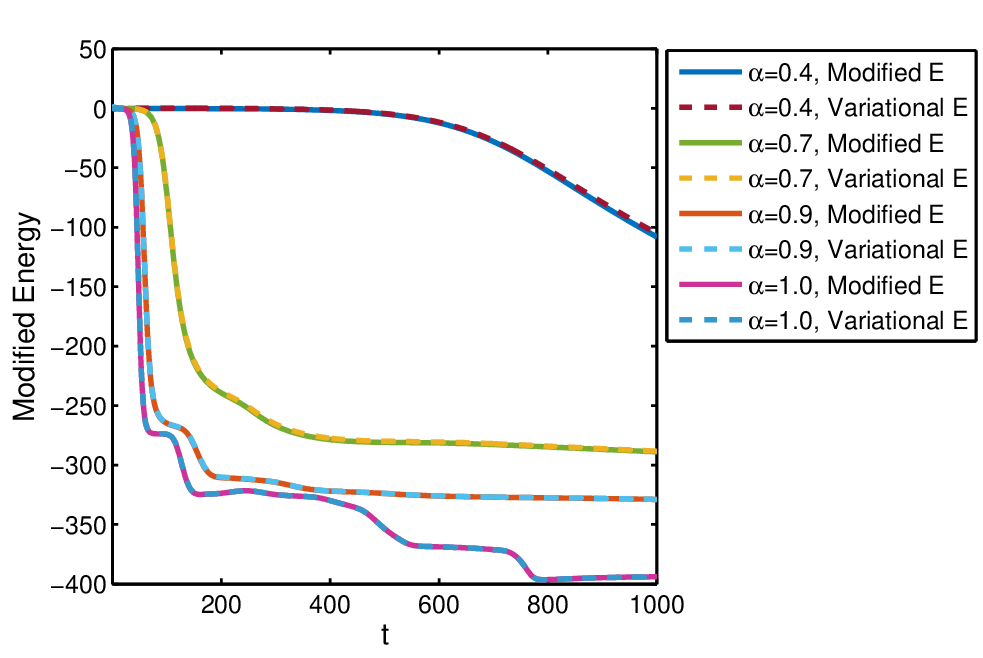}
      \caption{}
      \label{fig:SH_Modi_Energy}
    \end{subfigure}
    \hfill
    \begin{subfigure}[b]{0.48\textwidth}
      \centering
      \includegraphics[width=\textwidth]{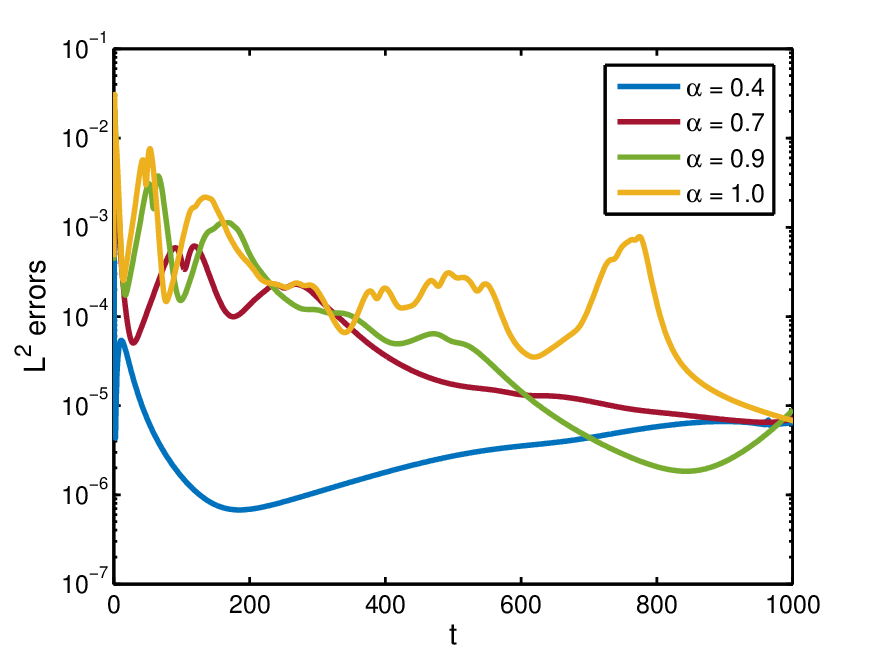}
      \caption{}
      \label{fig:SH_L2errors_r}
    \end{subfigure}
  \caption{The time-fractional Swift-Hohenberg equation. (a) The original discrete energy and the modified discrete energy. (b) The modified discrete energy and the variational discrete energy. (c) The $L^2$-error between the auxiliary variable $r^{n+\frac{1}{2}}$ and the original variable $\frac{1}{2}\left(\phi^{n+\frac{1}{2}}\right)^2    - \frac{g}{3} \phi^{n+\frac{1}{2}} +c_1-S$.}
  \label{fig:SH}
\end{figure}

\begin{figure}[H]
    \centering
    \includegraphics[width=1\textwidth ]{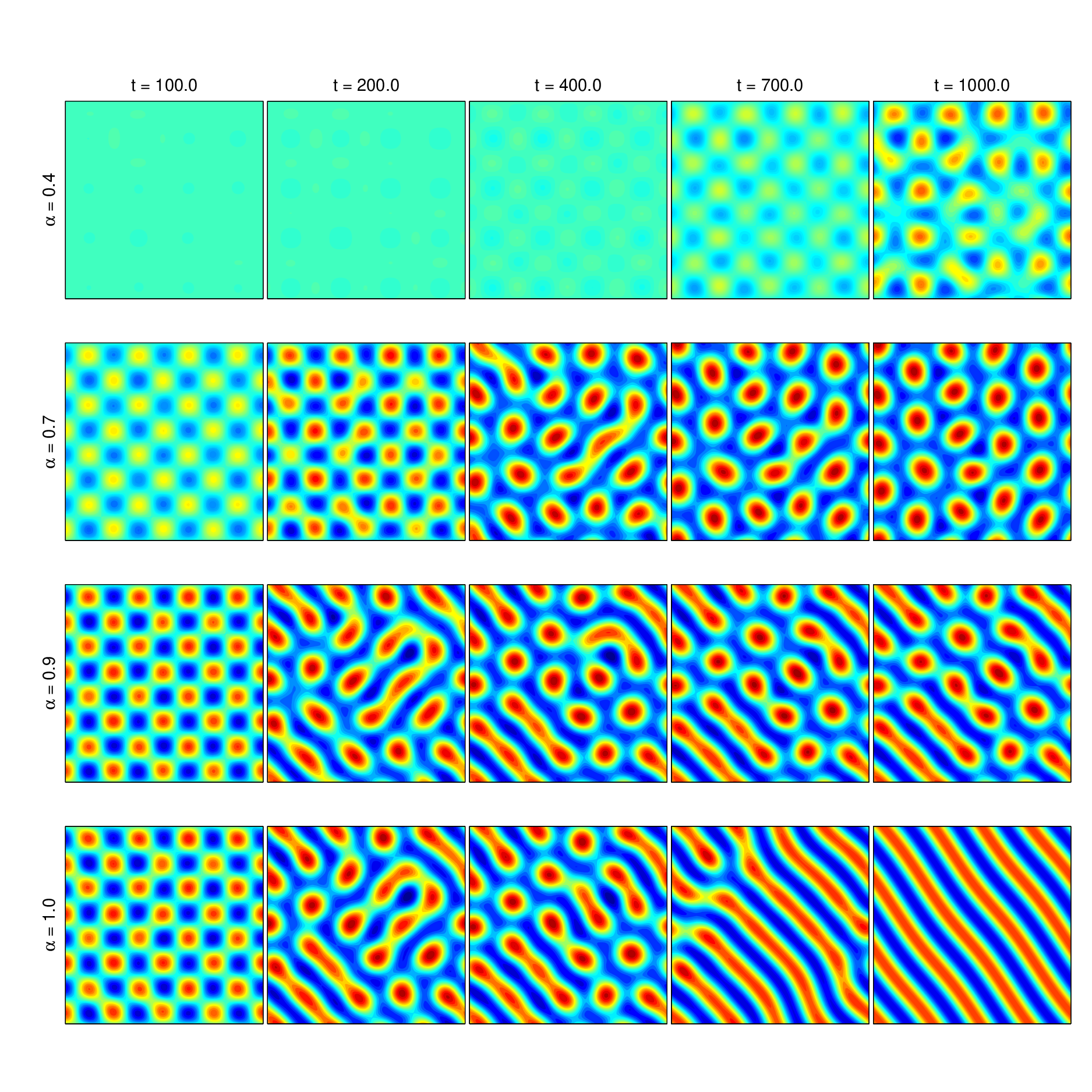}
    \caption{The snapshots of the time-fractional Swift-Hohenberg equation with different $\alpha$.}
    \label{fig:SH_Snapshots}
\end{figure}

\section{Conclusions and remarks}
\label{section5}
In this paper, we develop linear relaxation schemes for solving the time-fractional volume-conserved Allen-Cahn, Cahn-Hilliard, and Swift-Hohenberg equations. Auxiliary variables are introduced to approximate the original nonlinear term. Time-stepping is then carried out numerically by directly solving the algebraic equation associated with the expression of the auxiliary variable, which prevents the error accumulated over time as observed in IEQ and SAV methods, where the algebraic equation is replaced by its time derivative. We prove the energy stability and the variational energy dissipation property of the numerical schemes. Numerical results demonstrate the effectiveness of the proposed numerical schemes. In our future work, we shall consider using the ideas presented in this paper to design efficient structure-preserving schemes for phase-field problems coupled with fluid dynamics.

\section*{Acknowledgments}
H. Yu is partially supported by the China Scholarship Council for one year of research at the University of Dundee. Z. Wang is partially supported by the China Postdoctoral Science Foundation under grant 2024M760239. P. Lin is partially supported by the National Natural Science Foundation of China 12371388, 11861131004.

\bibliographystyle{elsarticle-num}
\bibliography{Ref}
\end{document}